\documentclass[a4paper,11pt]{article}
\usepackage{epsfig,amsmath,amsfonts,breqn}

\usepackage[T1]{fontenc}
\usepackage{dsfont}
\usepackage[latin1]{inputenc}  % direkte Eingabe von Umlauten

\usepackage{amsthm}

\usepackage{color}
\usepackage{paralist}
\usepackage{graphicx,psfrag}
 
\usepackage{tikz}

\usepackage{scalefnt}

\usepackage{rotating}
\usepackage{fix-cm}

\newcommand{\legendlineb}[3]{\tikz[baseline=-3pt]{\path[draw=#1,#2,line width=#3](0,0)--(0.63,0);}}

\usepackage{natbib}

\def\XXint#1#2#3{{\setbox0=\hbox{$#1{#2#3}{\int}$ }
		\vcenter{\hbox{$#2#3$ }}\kern-.6\wd0}}

\theoremstyle{definition}
\theoremstyle{plain}
\newtheorem{theorem}{Theorem}%[section]
\theoremstyle{plain}
\theoremstyle{plain}
\newtheorem{corollary}{Corollary}%[section]

\theoremstyle{plain}
\newtheorem{remark}{Remark}%[section]

\newcommand{\R}{\mathds{R}}
\newcommand{\N}{\mathds{N}}

\newcommand{\sn}{\mathrm{sn}}
\newcommand{\cn}{\mathrm{cn}}
\newcommand{\dn}{\mathrm{dn}}

\begin{document}
\begin{center}
{\LARGE \bf First Passage Time of Nonlinear Diffusion Processes with Singular Boundary Behavior} \\[3mm]
  {\large L. Dostal$^1$ and N. Sri Namachchivaya$^{2}$}\\[2mm]
$^1$\small Institute of Mechanics and Ocean Engineering, Hamburg University of Technology, Hamburg, Germany\\
$^2$ Department of Applied Mathematics, University of Waterloo, Waterloo, Ontario, Canada \label{firstpage}
\end{center}

\begin{abstract}
New theorems for the moments of the first passage time of one dimensional nonlinear stochastic processes with an entrance boundary $x_e$ are formulated. This important class of one dimensional stochastic processes results among others from approximations of the energy or amplitude of second order nonlinear stochastic differential equations.
Since the diffusion of a stochastic process vanishes at an entrance boundary, $x_e$ is called a singular point of the stochastic process. The theorems for the moments of the first passage times are validated based on existing analytical results. In addition, the first passage times of a nonlinear stochastic differential equation, which is important for the determination of dangerous ship roll dynamics, are calculated. The proposed analytical expressions for the moments of the first passage times can be calculated very fast using standard quadrature formulas.
\end{abstract}
\begin{center}
\small \textbf{Keywords: first passage time; random
excitation; stochastic dynamics, Duffing oscillator} \normalsize
\end{center}

\section{Introduction}
A very important problem in the field of diffusion processes and random dynamical systems deals with the determination of  times of the considered nonlinear diffusion process to reach certain boundaries for the first time, starting at a prescribed initial condition.

Using approximation techniques such as stochastic averaging, as described for example in \cite{namachchivaya:1991,roberts1978first,dostal:2012}, many nonlinear problems of higher order can be approximated  by one dimensional It\^o stochastic differential equations for the total energy or the response amplitude. 
Such one dimensional stochastic processes have singular boundaries where the diffusion becomes zero.

This makes it necessary to analyze the limiting behavior of these diffusion processes at such boundaries, since standard formulas or numerical calculations can not be applied appropriately without considering these singularities.
One dimensional diffusion processes of this kind arise from approximations of the energy or amplitude of important second order nonlinear problems, for example dynamical systems which undergo a co-dimension two bifurcation \cite{namachchivaya:1991}, oscillators with nonlinear damping \cite{roberts1978first}, pendulum or ship roll dynamics \cite{roberts:2000,dostal:2012,dostal:2016,dostal:2017b}.

In many applications it is important to make sure that the response process of a system does not leave a safe domain during its operation time. Then the first time the response process leaves the safe domain has to be determined.
Being able to establish formulas for such first passage times, leads to a major impact on the determination of the system reliability and on the expected duration before system failures occur.

After introducing the necessary theory of diffusion processes, we state our new theorems for the moments of the first passage time in section \ref{sec:Firstpassagetime}, followed by three example applications of the presented theory in sections \ref{sec:example} to \ref{sec:exampleParametric}.

In these sections the first passage time of a Linear system with external excitation is analyzed first, in order to illustrate the theory. The obtained results are validated using previously obtained results by \cite{ariaratnam:1976}. Then results of the presented theory for a forced and damped Mathieu oscillator are analyzed and validated. The Mathieu oscillator is widely used for the modeling of parametrically excited physical systems and has been the subject of many studies, for example \cite{ariaratnam:1976} \cite{vanvinckenroye2017average}.
Finally, results are shown for a Duffing oscillator forced by external and parametric excitation.
This nonlinear system has many applications. It is for example important for the dynamics and capsizing analysis of ships \cite{dostal:2012} and for the analysis of energy harvesting \cite{Yurchenko:2013,dostal:2017b}.

\section{One dimensional diffusion process}
For many problems it is necessary to consider a time homogeneous regular diffusion $\{x(t),t>0\}$ on the interval $D=(x_l;x_c]$, which satisfies the Itô stochastic differential equation
\begin{equation}\label{eq:Ito}
\mathrm{d}x(t)=m(x(t),t)\,\mathrm{d}t+\sigma(x(t),t)\,\mathrm{d}W,
\end{equation}
where $m: \R \times \R \rightarrow \R,$ and $\sigma: \R \times \R \rightarrow \R$ are measurable functions, see e.g. \cite{Oksendal:1992}.

As shown in the previous section, such processes result for example from stochastic averaging of general nonlinear oscillators.

The differential generator $\mathcal{L}$ of the diffusion $x$ defined by the It\^{o} equation \eqref{eq:Ito} is
given by
\begin{equation}\label{eq:generator}
\mathcal{L}p(x)=m(x)\frac{\partial}{\partial x}p(x)+\frac{1}{2}\sigma^2(x)\frac{\partial^2}{\partial x^2}p(x)
\end{equation}
and the associated adjoint operator is
\begin{equation}\label{eq:AdjointGenerator}
\mathcal{L^*}p(x)=-\frac{\partial}{\partial x} \left(m(x)p(H)\right)+\frac{1}{2}\frac{\partial^2}{\partial x^2}\left(\sigma^2(x)p(x)\right).
\end{equation}
We introduce the scale measure
\begin{equation}
S(x)=\int^{x}s(x)dx,
\end{equation}
where
\begin{equation}\label{eq:scaling dense}
s(y)=\exp\left(-2\int^{y}\frac{m(\theta)}{\sigma^2(\theta)}d\theta\right)\notag
\end{equation}
is the scale density.
After defining the speed density
\begin{equation}\label{eq:speeddensity}
\mu(y)=1/(\sigma^2(y) s(y))
\end{equation}
and the speed measure $M$ by
\begin{equation}
dM(y)=\mu(y)dy,
\end{equation}
we can transform the operator $\mathcal{L}$ to
\begin{equation}
\mathcal{L}p(x)=\frac{1}{2}\frac{\partial}{\partial M}\left(\frac{\partial p(x)}{\partial S} \right),
\end{equation}
such that the drift is identically zero. A discussion on the meaning of scale density $s$ and speed density
$\mu$ can be found in \cite{karlin:1981}.

The stationary probability density function
$p_{st}$ associated with \eqref{eq:Ito}, is the solution of
the Fokker-Planck equation
\begin{equation}\label{eq:FPKeq}
\mathcal{L}^*p_{st}(x)=0,
\end{equation}
where $\mathcal{L}^*$ is the adjoint of the operator  $\mathcal{L}$.
The solution of \eqref{eq:FPKeq} can be given in terms of
\begin{equation}
p_{st}(x)=\mu(x)[c_1S(x)+c_2].
\end{equation}
The coefficients $c_1$ and $c_2$ are determined by the boundary and
normality conditions.\\
\\
Following \citep{karlin:1981} the (left)
boundaries can be classified as
\begin{itemize}
\item Entrance, if
\begin{equation}\label{eq:EingangsrandbedingungKarlin}
\Sigma_l(x_l)=\infty\;\;  \mathrm{and} \;\; N_l(x_l)<\infty,
\end{equation}
\item reflecting, if
\begin{equation}\label{eq:reflektierenderandbedingungKarlin}
\Sigma_l(x_l)<\infty,\;\; N_l(x_l)<\infty \;\;  \mathrm{and} \;\; M(x_l)=0,
\end{equation}
\item exit, if
\begin{equation}\label{eq:AusgangrandbedingungKarlin}
\Sigma_l(x_l)<\infty  \;\;  \mathrm{und} \;\;M(x_l,x]=\infty.
\end{equation}
\end{itemize}
Here, $\Sigma_l(x_l)$ is the time
to reach the left boundary $x_l$ starting at $x_0\in[x_l,x_c]$,
whereas $N_l(x_l)$ is the time to reach $x_0\in(x_l,x_c]$ starting at
$x_l$. These measures are given by
%\begin{equation}
\begin{equation}\label{eq:N_l}
N_l(x_l)=\int_{x_l}^{x_0}\left\{\int_{z}^{x_0} s(y)dy\right\}
\mu(z)dz=\int_{x_l}^{x_0} S[z,x_0]\, \mu(z)\,\mathrm{d}z,
\end{equation}
and
\begin{equation}\label{eq:Sigma_l}
\Sigma_l(x_l)=\int_{x_l}^{x_0}  \left\{\int_{x_l}^{z}s(y)dy\right\}
 \mu(z)dz=\int_{x_l}^{x_0} S(x_l,z]\, \mu(z)\,\mathrm{d}z,
\end{equation}
where
\begin{equation}
S(x_l,z]:=\lim\limits_{a\downarrow x_l}  S[a,z],\notag
\end{equation}
with the definition
\begin{equation}
S[a,b]:=\int_{a}^{b}s(y)\, \mathrm{d}y , \quad a,b \in \R. \notag
\end{equation}
If we assume an entrance boundary at $x_e$ and a reflecting boundary at $x_c$ for the process $x(t)\in[x_e;x_c]$, then $c_1=0$.
In this case a stationary solution of the
Fokker-Planck equation \eqref{eq:FPKeq} exists and is given by
\begin{equation}\label{eq:statDensity}
p_{st}(x)=\frac{c_2}{\sigma^2(x)}\exp\left(2\int_{x_e}^x\frac{m(x)}{\sigma^2(x)}\right).
\end{equation}

\section{First passage time}\label{sec:Firstpassagetime}
%%%%%%%%%%%%%%%%%%%%%% first passage times:%%%%%%%%%%%%%%%%%%%%%
We are now looking for the mean time until the process $x(t)$ reaches certain
values $x_c$ starting at an initial value $x_0$.
Let $\triangle \in D=(x_l,x_c)$ be an inner point.
The first passage time of $x_c$ starting at $ x_0, \; \triangle <x_0 <x_c, $ at time $ t_0 $ is defined by $T_{x_c}(x_0)=\inf\{ t:x(t)=x_c|x(t_0)=x_0 \}$.
The $n$-th moment $M_n (x_0) $ of the first passage time $T_{x_c}(x_0)$
is given by the generalized Pontryagin equation
\begin{equation}\label{eq:PontryaginGleichung}
\mathcal{L}\,M_n(x_0)=-n\, M_{n-1}(x_0),\;\;\; M_n(x_c)=0,
\end{equation}
with $M_0(x_0):=1$. This equation can be  solved for $ n = 1 $ by
\begin{equation}\label{eq:1th_moment_exittime}
M_1(x_0)=2 \left [ 1-\frac{S[x_0,x_c]}{S[\triangle, x_c]}   \right]\int_{x_0}^{x_c} S[z, x_c]\,\mu(z)\,\mathrm{d}z+2 \,  \frac{S[x_0,x_c]}{S[\triangle, x_c]}   \int_{\triangle}^{x_0} S[\triangle,z]\, \mu(z)\,\mathrm{d}z.
\end{equation}
This formula is equivalent to the formula from \cite{karlin:1981} on page 197.\\
\\
In the following, we consider the singular case, for which the diffusion at the left boundary $x_l$ disappears, i.e. $\sigma^2 (x_l) = 0 $. The point $ x_l $ is then called a singular point of the diffusion defined by the generator $ \mathcal{L} $ . It is obvious that due to the singularity in equation \eqref{eq:scaling dense} at $\sigma^2 (x_l) = 0 $, the limit behavior of equations \eqref{eq:speeddensity} and \eqref{eq:1th_moment_exittime} has to be determined.\\
\\
For the singular case the mean first passage time $ M_1 (x_0) $ from equation \eqref{eq:1th_moment_exittime} can be determined.
This yields our main result in the next Theorem.
\begin{theorem}(Main Theorem) \label{th:meanexittime}\\
Let $x_c$ be a regular or an exit boundary of the diffusion process $x(t)$ as defined by the Itô equation \eqref{eq:Ito}.
Let further\\
(i) $x_l$ be an entrance boundary of the process $x(t)$\\
or\\
(ii) $S(x_l,z]=\infty$ and $M(x_l,z]$ $\forall x_l<z\leq x_c$.\\
Then the mean time until $x(t)$ reaches $x_c$ starting
at any $x_0\in[x_l,x_c]$ is given by
\begin{equation}\label{eq:meanexittime}
M_1(x_0)=2\int_{x_0}^{x_c}\left [\int_{z}^{x_c}s(y)\, \mathrm{d}y\right ] \mu(z)dz + 2\,\int_{x_0}^{x_c}s(y)\, \mathrm{d}y \int_{x_l}^{x_0}\mu(z)dz
\end{equation}
and $M_1(x_0)<\infty$ $\forall\; x_0\in[x_l;x_c]$.
\end{theorem}
\begin{proof}
We have to show, that the limit as $\triangle \rightarrow x_l$ of equation \eqref{eq:1th_moment_exittime} is given by the resulting equation \eqref{eq:meanexittime} of Theorem \ref{th:meanexittime}.
For the case (i) in which the singular point $x_l$ is an entrance boundary, we first observe that due to the definition \eqref{eq:N_l} of $N_l(x_l)$ we have
\begin{equation}\label{eq:muFinite}
\lim\limits_{\triangle \downarrow x_l}{ \int_{\triangle}^{z} \mu(\eta)\,\mathrm{d}\eta <\infty},\;\forall x_l<z\leq x_c
\end{equation}
since $N_l(x_e)<\infty$. Then
\begin{equation}
\Sigma_l(x_l) + N_l(x_l) = \lim\limits_{\triangle \downarrow x_l}{ \{S[\triangle, z]\int_{\triangle}^{z} \mu(\eta)\,\mathrm{d}\eta} \}
\end{equation}
implies the limit
\begin{equation}
\lim\limits_{\triangle \downarrow x_l}{ S[\triangle, z]}=S(x_l, z]=\infty, \;\forall x_l<z\leq x_c.
\end{equation}
Thus case (i) implies case (ii). Using conditions (ii), we can calculate the limit as $\triangle\rightarrow x_l$ of the first term in equation \eqref{eq:1th_moment_exittime} by
\begin{equation}
\lim\limits_{\triangle \downarrow x_l}{2 \left [ 1-\frac{S[x_0,x_c]}{S[\triangle, x_c]}   \right]\int_{x_0}^{x_c} S[z, x_c]\,\mu(z)\,\mathrm{d}z=2\int_{x_0}^{x_c}S[z,x_c]\mu(z)dz}.
\end{equation}
The last term of equation \eqref{eq:1th_moment_exittime} involves the measure
\begin{equation}
\Sigma_l(x_l)=\lim\limits_{\triangle \downarrow x_l} \int_{\triangle}^{x_0} S[\triangle,z]\, \mu(z)\,\mathrm{d}z  =\infty.\notag
\end{equation}
Therefore, that last term will not vanish in general, although $S(x_l, x_c]=\infty$. The corresponding limit $Z$ is given by
\begin{equation}\label{eq:limitSecondTerm}
\begin{aligned}
Z:=&\lim\limits_{\triangle \downarrow x_l} 2\frac{S[x_0,x_c]}{S[\triangle, x_c]} \int_{\triangle}^{x_0} S[\triangle,z]\, \mu(z)\,\mathrm{d}z.
\end{aligned}
\end{equation}
Because the drift $m$ and the diffusion $\sigma$ are measurable functions, we can interchange the limit with the integral and obtain
\begin{equation}
\begin{aligned}
Z&=2\,S[x_0,x_c] \int_{x_l}^{x_0} \lim\limits_{\triangle  \downarrow x_l} \frac{S[\triangle,z]}{S[\triangle, x_c]}\, \mu(z)\,\mathrm{d}z\\
&=2\,S[x_0,x_c] \int_{x_l}^{x_0} \lim\limits_{\triangle \downarrow x_l}\frac{S[\triangle,z]}{S[\triangle, z]+S[z, x_c]}\, \mu(z)\,\mathrm{d}z\\
&=2\,S[x_0,x_c] \int_{x_l}^{x_0} \lim\limits_{\triangle \downarrow x_l}\frac{1}{1+\frac{S[z, x_c]}{S[\triangle,z]}}\, \mu(z)\,\mathrm{d}z.
\end{aligned}
\end{equation}
Since $S(x_e, x_c]=\infty$ it follows that
\begin{equation}
\begin{aligned}
Z=2\,S[x_0,x_c] \int_{x_l}^{x_0} \mu(z)\,\mathrm{d}z,
\end{aligned}
\end{equation}
which is the last term of equation \eqref{eq:meanexittime}.

Due to $N_l(x_l)<\infty$, the process $x$ reaches the point $x_c$ in finite time with probability one, starting at an arbitrary $x_0\in[x_l;x_c]$. In addition, the last term of equation \eqref{eq:meanexittime} is also finite, because of equation \eqref{eq:muFinite}.
These facts imply $M_1(x_0)<\infty$ $\forall\; x_0\in[x_l;x_c]$.
%\begin{flushright}$\Box$\end{flushright}
\end{proof}
With this result, we are able to calculate the mean first passage time of nonlinear diffusion processes.

\begin{corollary}\label{th:meanexittimeCorollary}
Let $x_l$ be an entrance boundary %(i.e. $\Sigma_l(x_e)=\infty$ and $N_l(x_e)<\infty$)
and let $x_c$ be a regular or an exit boundary of the diffusion process $x(t)$ as defined by the Itô equation \eqref{eq:Ito}. %(i.e.$\Sigma_r(x_c)<\infty$),
Then the mean time until $x(t)$ reaches $x_c$ starting
at $x_0=x_l$ is finite and given by
\begin{equation}\label{eq:meanexittimeCorollary}
M_1(x_0)=2\int_{x_0}^{x_c}S[z,x_c]\mu(z)dz.
\end{equation}
\end{corollary}
\begin{proof}
The assertion follows from Theorem \ref{th:meanexittime} because of
\begin{equation}
\begin{aligned}
\lim\limits_{x_0 \downarrow x_l} 2\,S[x_0,x_c] \int_{x_l}^{x_0}\mu(z)dz=0.
\end{aligned}
\end{equation}
\end{proof}
We now turn to the determination of moments of the mean first passage time $T_{x_c}(x_0)$ of the process from equation \eqref{eq:Ito}.
Starting from a point $x_0\in[\triangle;x_c]$, the $n$-th moment for reaching the regular boundary $ x_c $ can be obtained from the generalized Pontryagin equation \eqref{eq:PontryaginGleichung}. This moment is given by
%
%Ausgehend von einem Startpunkt $x_0\in[\triangle;x_c]$ folgt aus der generalisierten Pontryagin-Gleichung \eqref{eq:PontryaginGleichung}\\
%\\
%das $n$-te Moment für das Erreichen des regulären Randes $x_c$. Dieses lässt sich schreiben als
\begin{equation}\label{eq:nth_moment_exittime}
\begin{aligned}
M_n(x_0)&=2\, n \left [ 1-\frac{S[x_0,x_c]}{S[\triangle, x_c]}   \right]\int_{x_0}^{x_c} S[z, x_c]\,\mu(z)\,M_{n-1}(z)\,\mathrm{d}z\\
&\quad +2\, n \, \frac{S[x_0,x_c]}{S[\triangle, x_c]}   \int_{\triangle}^{x_0} S[\triangle,z]\, \mu(z)\,M_{n-1}(z)\,\mathrm{d}z,
\end{aligned}
\end{equation}
where $M_0(x_0):=1$, cf. \cite{karlin:1981} on page 197.\\
\\
For the singular case $\sigma^2(x_e)=0$ at the entrance boundary $x_e$, we can extend Theorem \ref{th:meanexittime} for the $n$-th moment of the first passage time  $T_{x_c}(x_0)$ as follows.
\begin{theorem}\label{satz:momentsOfFirstPassage}%\emph{(Dostal)}\\
Let $x_e$ be an entrance boundary %(i.e. $\Sigma_l(x_e)=\infty$ and $N_l(x_e)<\infty$)
and let $x_c$ be a regular or an exit boundary of the diffusion process $x(t)$ as defined by the Itô equation \eqref{eq:Ito}. %(i.e.$\Sigma_r(x_c)<\infty$),
Let further $M_{n}(x_0)$ be the $n$-th moment of the first passage time until $x(t)$ reaches $x_c$ starting
at any $x_0\in[x_e,x_c]$. Then
\begin{equation}\label{eq:Satz_nth_moment_exittime}
\begin{aligned}
M_n(x_0)&=2\,n\int_{x_0}^{x_c}\left[ \int^{x_c}_z s(y)\,\mathrm{d}y\right]M_{n-1}(z)\,\mu(z)\,\mathrm{d}z \\
&\hspace{2cm}+ 2\,n\,\int_{x_0}^{x_c}s(y)\, \mathrm{d}y  \int_{x_e}^{x_0}M_{n-1}(z)\mu(z)dz,
\end{aligned}
\end{equation}
where the first moment  $M_1(x_0)$ of the recursion is obtained from Theorem \ref{th:meanexittime}
and the $n$-th moment is finite, i.e. $|M_n(x_0)|<\infty$ $\forall n \in \mathds{N}$.
\end{theorem}
\begin{proof}
Let $|M_{n-1}(x_0)|<\infty$. Then a constant $C_n\in\mathds{R^+}$ exists for any $n$, such that
\begin{equation}
|M_{n-1}(x_0)|\leq C_n<\infty.\notag
\end{equation}
It follows that $|M_{n}(x_0)|<\infty$ as well, since
%Daraus folgt, dass auch $|M_{n}(x_0)|<\infty$ ist, denn es gilt
\begin{equation}
\begin{aligned}
|M_{n}(x_0)|&=2\, n \,\bigg|\lim\limits_{\triangle \downarrow x_e}\bigg\{ \left [ 1-\frac{S[x_0,x_c]}{S[\triangle, x_c]}   \right]\int_{x_0}^{x_c} S[z, x_c]\,\mu(z)\,M_{n-1}(z)\,\mathrm{d}z\\
&\hspace{2cm}+\frac{S[x_0,x_c]}{S[\triangle, x_c]} \int_{\triangle}^{x_0} S[\triangle,z]\, \mu(z)\,M_{n-1}(z)\,\mathrm{d}z\bigg\}\bigg|\\
&\leq2\, n \,C_n\,  \bigg|\lim\limits_{\triangle \downarrow x_e}\bigg\{\left [ 1-\frac{S[x_0,x_c]}{S[\triangle, x_c]}  \right] \int_{x_0}^{x_c} S[z, x_c]\,\mu(z)\,\mathrm{d}z\\
&\hspace{2cm}+ \frac{S[x_0,x_c]}{S[\triangle, x_c]}   \int_{\triangle}^{x_0} S[\triangle,z] \,\mu(z)\,\mathrm{d}z\bigg\}\bigg|\\
&=n\,C_n\,|M_1(x_0)|<\infty.\notag
\end{aligned}
\end{equation}

Because of $M_1(x_0)<\infty$, we have by mathematical induction that\\ $M_n(x_0)<\infty$ $\forall n\in \mathds{N}$.
Since $M_n(x_0)<\infty$, it follows that the $n$-th moment of the first passage time $T_{x_c}(x_0)$ is given by equation \eqref{eq:Satz_nth_moment_exittime}. The calculations are analog to the calculations in the proof of Theorem \ref{th:meanexittime}.
\end{proof}
We can again consider the first passage time of $x_c$ starting at the entrance boundary $x_e$.
\begin{corollary}\label{th:momentsexittimeCorollary}
Let $x_e$ be an entrance boundary %(i.e. $\Sigma_l(x_e)=\infty$ and $N_l(x_e)<\infty$)
and let $x_c$ be a regular or an exit boundary of the diffusion process $x(t)$ as defined by the Itô equation \eqref{eq:Ito}. %(i.e.$\Sigma_r(x_c)<\infty$),
Then the moments of the time until $x(t)$ reaches $x_c$ starting
at $x_0=x_e$ is finite and given by
\begin{equation}\label{eq:meanexittimeCorollary}
M_{n}(x_0)=2\,n\int_{x_0}^{x_c}\left[ \int^{x_c}_z s(y)\,\mathrm{d}y\right]M_{n-1}(z)\,\mu(z)\,\mathrm{d}z.
\end{equation}
\end{corollary}
\begin{proof}
The assertion follows from Theorem \ref{satz:momentsOfFirstPassage}, since $M_{n}(x_0)<\infty\; \forall n \in \N$,  and
\begin{equation}
\begin{aligned}
\lim\limits_{x_0 \downarrow x_e} 2\,n\,S[x_0,x_c] \int_{x_e}^{x_0}M_{n-1}(z)\mu(z)dz=0.
\end{aligned}
\end{equation}
\end{proof}

\section{Stochastic averaging}
Important systems can be modeled by 
the perturbed Hamiltonian system 
 \begin{equation}\label{eq:hamiltonian}
 \begin{aligned}
 \frac{\mathrm{d}}{\mathrm{d} t}x&=\frac{\partial H(x,y)}{\partial y},\\
 \frac{\mathrm{d}}{\mathrm{d} t}y&= -\frac{\partial H(x,y)}{\partial x}-\varepsilon d(y) +\sqrt{\varepsilon}f(x,\boldsymbol{\xi}_t)
 \end{aligned}
 \end{equation}
 with the two dimensional state space $\mathbf{Z}^{\varepsilon}:=(x,y)\in D\subset \R^2$. 
Thereby the function $H(x,y)$ is the Hamiltonian, $d(y)$ is a damping function, $f(x,\boldsymbol{\xi}_t)$ is a function of random excitations and $\varepsilon > 0$.
 For the case of weakly perturbed systems of type \eqref{eq:hamiltonian} with small $\varepsilon \ll 1$, a stochastic averaging method is proposed in the following Theorem, using results by \cite{Khasminskii:1968}, \cite{borodin:1977} and \cite{FreidlinBorodin:1995}.
 With this method, the stochastic process of the Hamiltonian $H$ can be obtained, which is the process of total energy of the corresponding nonlinear oscillator.
 Such a stochastic averaging procedure was used in \cite{dostal:2012} and can be generalized as follows.
 \begin{theorem}\label{eq:satzDostalX}%\paragraph{\textbf Proposition 1}
 	Let $Z^{\varepsilon}=(X ,Y)\in D\subseteq \R^2$ be the solution of the SDE
 	\begin{eqnarray}
 	\frac{\mathrm{d}X }{\mathrm{d}t}&=&\varepsilon f_1(X,Y)+\sqrt{\varepsilon} \,\mathrm{\mathbf{f_0}}(X,Y)\:\boldsymbol{\xi},\label{eq:satzDostalX_gl_1}\\
 	\frac{\mathrm{d}Y}{\mathrm{d}t}&=&g(X,Y), \hspace{8mm} (X(0),Y(0))=(x_0,y_0) \in D,\;\;\; \varepsilon> 0,
 	\end{eqnarray}
 	and let the following conditions be fulfilled:
 	\begin{itemize}[i)]
 		\item
 		The stochastic process $\boldsymbol{\xi}_t:=\boldsymbol{\xi}(\omega,t)=(\xi_1(t),\ldots,\xi_k(t))^{\mathrm{T}}\in\mathds{R}^k$ is stationary, absolutely regular with sufficient mixing properties, and with $E\{\xi_j(t)\}=0$.
 		\item  The functions $\mathbf{\mathrm{\mathbf{f_0}}}$ and $f_1$ satisfy certain limits in order to ensure uniqueness of the solution.
 		\item Without loss of generality, the functions $\mathbf{\mathrm{\mathbf{f_0}}}:D\mapsto\mathds{R}^k$, $f_1:D\mapsto\mathds{R}$ and the solution of the equation ${\dot{Y}=g(x,Y)}$ are periodic with period $ T(x) $ for fixed $ x $.
 	\end{itemize}
 	Let the averaging operator  $\mathbb{M}$ for periodic functions $f:\mathds{R^+}\mapsto\mathds{R}$ with period $T$ be defined by
 	\begin{equation}\label{eq:Mittelungsopperator}
 	\begin{aligned}
 	\mathbb{M}\left\{f\right\}=&\frac{1}{T}\int_0^Tf(t)\mathrm{d}t
 	\end{aligned}
 	\end{equation}
 	and let $Y^{x}$ be the solution of the ordinary differential equation
 	\begin{align*}
 	\mathrm{d}Y^{x}=&g(X ,Y^{x})\:\mathrm{d}t,\:X =x,\;\;Y^{x}(0)=y.
 	\end{align*}
 	If the limits
 	\begin{align*}
 	m(Z)=&\mathbb{M}\bigg\{f_1(x,Y^x(t))\\
 	&+\int_{-\infty}^0 \mathrm{cov}\left(
 	\left\{\frac{\partial \mathbf{\mathrm{\mathbf{f_0}}}(X ,Y^x(t))\:\boldsymbol{\xi}_t}{\partial X }\right\}_{X =x},\mathbf{\mathrm{\mathbf{f_0}}}(x,Y^x(t+s))\:\boldsymbol{\xi}_{t+s}
 	\right)\mathrm{d}s\bigg\}, \notag \\
 	\sigma^2(Z)=&\mathbb{M}\bigg\{\int_{-\infty}^{\infty}\mathrm{cov}\left(\mathrm{\mathbf{f_0}}(x,Y^x(t))\:\boldsymbol{\xi}_{t}, \mathbf{\mathrm{\mathbf{f_0}}}(x,Y^x(t+s))\:\boldsymbol{\xi}_{t+s}\right)\mathrm{d}s\bigg\}%_{Z=y}
 	\end{align*}
 	exist, then the process $X(\tau)$, $\tau=\varepsilon t$, converges, as $\varepsilon \rightarrow
 	0$, weakly on the time interval of order
 	$\mathcal{O}(1/\varepsilon)$ to a diffusion Markov process $Z$
 	satisfying the It\^{o} stochastic differential equation
 	\begin{equation}\label{eq:averageItoGeneral}
 	\mathrm{d}Z(\tau)=m(Z)\mathrm{d}\tau+\sigma(Z)\mathrm{d}W_{\tau},\hspace{8mm} Z(0)=x_0,
 	\end{equation}
 	with the standard Wiener process $W_{\tau}$.
 \end{theorem}
 \begin{proof}
 	The deterministic solution $Y^{x}$ of equation $\mathrm{d}Y^{x}=g(x,Y^{x})\:\mathrm{d}t$ is determined for arbitrary but fixed $x$. Then, $Y$ is replaced by $Y^{x}$ in equation \eqref{eq:satzDostalX_gl_1}. The assertion follows by applying the Theorem from \cite{borodin:1977} for the resulting equation.
 \end{proof}
 If the functions $\mathbf{\mathrm{\mathbf{f_0}}}$ and $f_1$ are not periodic, then the procedure as described in \cite{Khasminskii:1966} has to be used.
 The essential result of Theorem \ref{eq:satzDostalX} is, that the total energy $H(\textbf{Z}^{\varepsilon}(t))$ of system \eqref{eq:hamiltonian} converges in probability at a scale $\mathcal{O}( t )$, $\tau=\varepsilon t $, to the diffusion Markov process $\bar{H}(t)$ as $\varepsilon \rightarrow 0$. The resulting stochastic process is given by the It\^{o} equation
 \begin{equation}\label{eq:Itoaveraged}
 \mathrm{d}\bar{H}(\tau)=m(\bar{H})\mathrm{d}\tau+\sigma(\bar{H})\mathrm{d}W_{\tau},
 \end{equation}
 where $W_{\tau}$ is the standard Wiener process. In order to simplify the notation, we will not distinguish between the original process $H$ and the averaged process $\bar{H}$.

\begin{remark}
In this work the function $g(X;Y)$ contained in Theorem~\ref{eq:satzDostalX} is chosen such that $g(X ,Y)=\sqrt{Q}$ with $Q:=2X - 2U(Y)$, whereby \mbox{$U(Y):\mathds{R}\mapsto\mathds{R}$} is a continuously differentiable function. Depending on the roots of $g(X,Y)$, the drift and diffusion coefficients have to be determined piecewise for different phase space regions. More details on the dependence of the nonlinear diffusion process on $g(X,Y)$ can be found in \cite{FreidlinWentzell:2012}.	
\end{remark}

\section{Linear system with external excitation}\label{sec:example}
In this section the mean and variance of first passage times is determined for amplitude crossing of a linear oscillator, which is subjected  to an additive stationary wide-band random excitation.
In the following sections the equations of motion of this oscillator are successively extended in order to demonstrate the developed theory.
The considered linear oscillator, as also studied in \cite{ariaratnam:1973}, is given by
\begin{equation}\label{eq:linear_oscillator_ariaratnam}
\ddot{x}+ \varepsilon 2 d \omega_n \dot{x} + \omega_n^2 x = \sqrt{\varepsilon} \xi(t).
\end{equation}
Here, the parameter $d$ is the damping ratio, $\omega_n$ is the undamped natural frequency, and $\xi(t)$ is a stationary stochastic process with zero mean and spectral density
\begin{equation}\label{eq:specDensity}
S_{\xi\xi}(\omega)=\frac{1}{2 \pi} \int_{-\infty}^{\infty}R_{\xi\xi}(s)\exp(\mathrm{i}\omega s)\mathrm{d}s,
\end{equation}
where $R_{\xi\xi}(s)=E[\xi(t)\xi(t+s)]$ is the autocorrelation of the stochastic process $\xi(t)$, and $E[\cdot]$ denotes the expected value. We introduce the non-dimensional variable
\begin{equation}
r(t)=\frac{a(t)}{\sqrt{\frac{1}{2}d \omega_n^3 S_{\xi\xi}(\omega_n)}},
\end{equation}
where $a(t)$ is the amplitude process of the oscillator \eqref{eq:linear_oscillator_ariaratnam}. Then according to \cite{ariaratnam:1973}, the Itô equation for the process $r(t)$ can be written as
\begin{equation}\label{eq:linear_oscillator_ariaratnam_r}
dr=  \varepsilon d \omega_n \left(\frac{1}{2 r}-r \right) dt +  \sqrt{\varepsilon d\omega_n}dW.
\end{equation}
Now the point $r=0$ is an \emph{entrance} boundary, since
\begin{equation}\label{eq:measures_r}
\begin{aligned}
S(0,x]=\lim\limits_{\triangle \downarrow 0} Ei(x^2)-Ei(\triangle^2)=\infty\;\mathrm{implying}\; \Sigma_l(0)=\infty,
\end{aligned}
\end{equation}
and
\begin{equation}\label{eq:measures_r2}
\begin{aligned}
N_l(0)=Ei(x^2)-\ln(x^2-\gamma)  <\infty\;\forall\, x\in (0,R), \,R>0.
\end{aligned}
\end{equation}
In equations \eqref{eq:measures_r} and  \eqref{eq:measures_r2}, the function $Ei(\cdot)$ represents the exponential integral and $\gamma$ is Euler's constant.
We now use the Theorems 1 and 2 in order to calculate the first two moments $M_1(r_0,R)$ and $M_2(r_0,R)$  of the first passage time until the process $r(t)$ reaches the value $R >0$ starting at $r_0\in[0,R]$. From these results the mean $M(r_0,R)=M_1(r_0,R)$ and variance $V(r_0,R)=M_2(r_0,R)-(M_1(r_0,R))^2$ of the considered first passage time are obtained.
In Figure \ref{fig:exitTimeWhiteComparison} the mean first passage
time of the linear oscillator with additive noise \cite{ariaratnam:1973} for reaching the boundary $R=2.2$ starting at $r_0$ is shown using Theorem 1 and Corollary 1 as well as the solution obtained by \cite{ariaratnam:1973}.
\begin{figure}[htbp]
\begin{center}
\includegraphics[width=11cm]{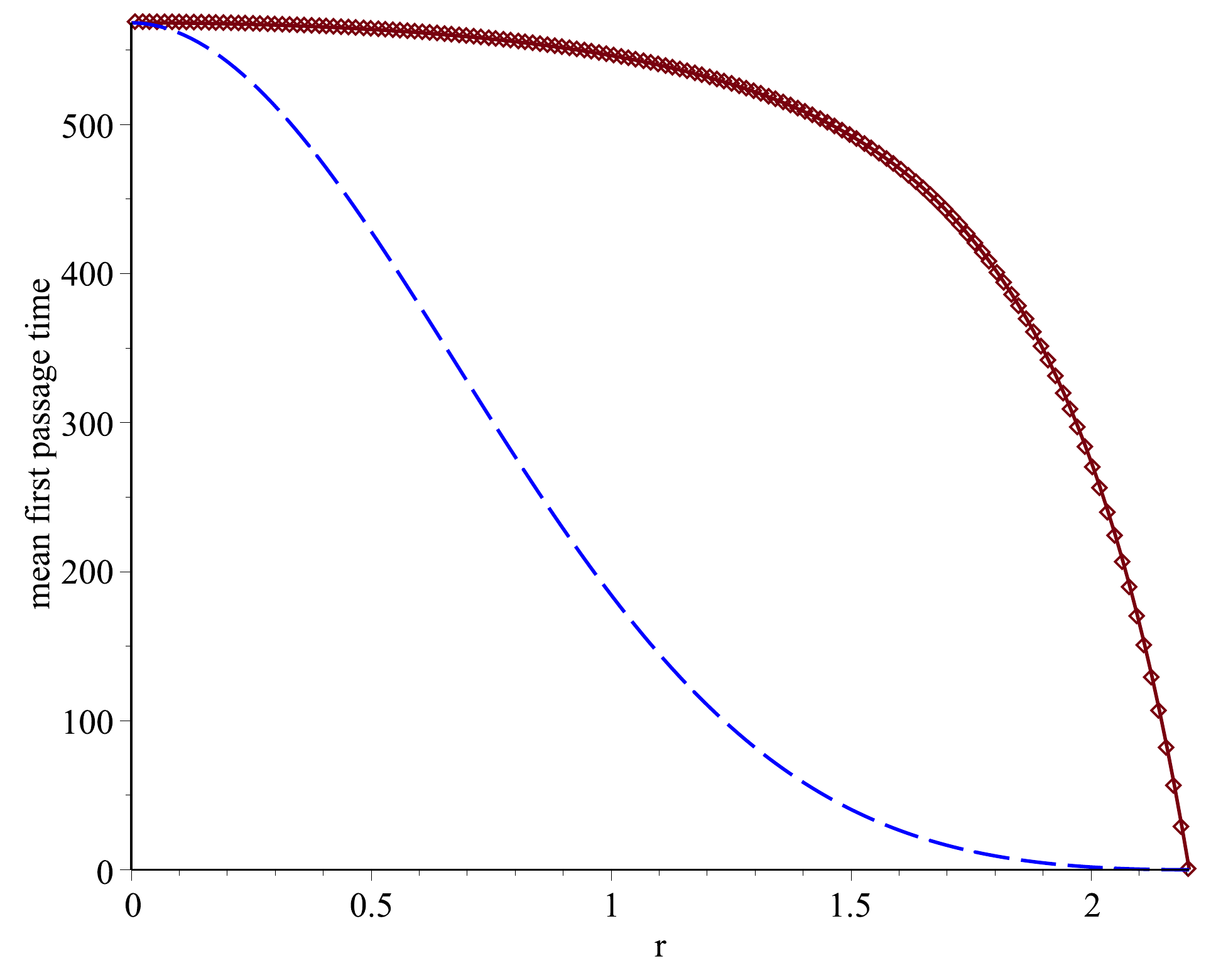}
\caption{Comparison of mean first passage times. Solution from Theorem~1~(\textcolor{red}{$\diamond$}), solution from Corollary~1(\textcolor{blue}{--  --}), solution from~\cite{ariaratnam:1973}~(\textcolor{red}{-}).}\label{fig:exitTimeWhiteComparison}
\end{center}
\end{figure}
The corresponding variance of the first passage time obtained from Theorem 2 and from the solution provided by \cite{ariaratnam:1973} is shown in Figure \ref{fig:VarianceexitTimeWhiteComparison}.
\begin{figure}[htbp]
\begin{center}
\includegraphics[width=11cm]{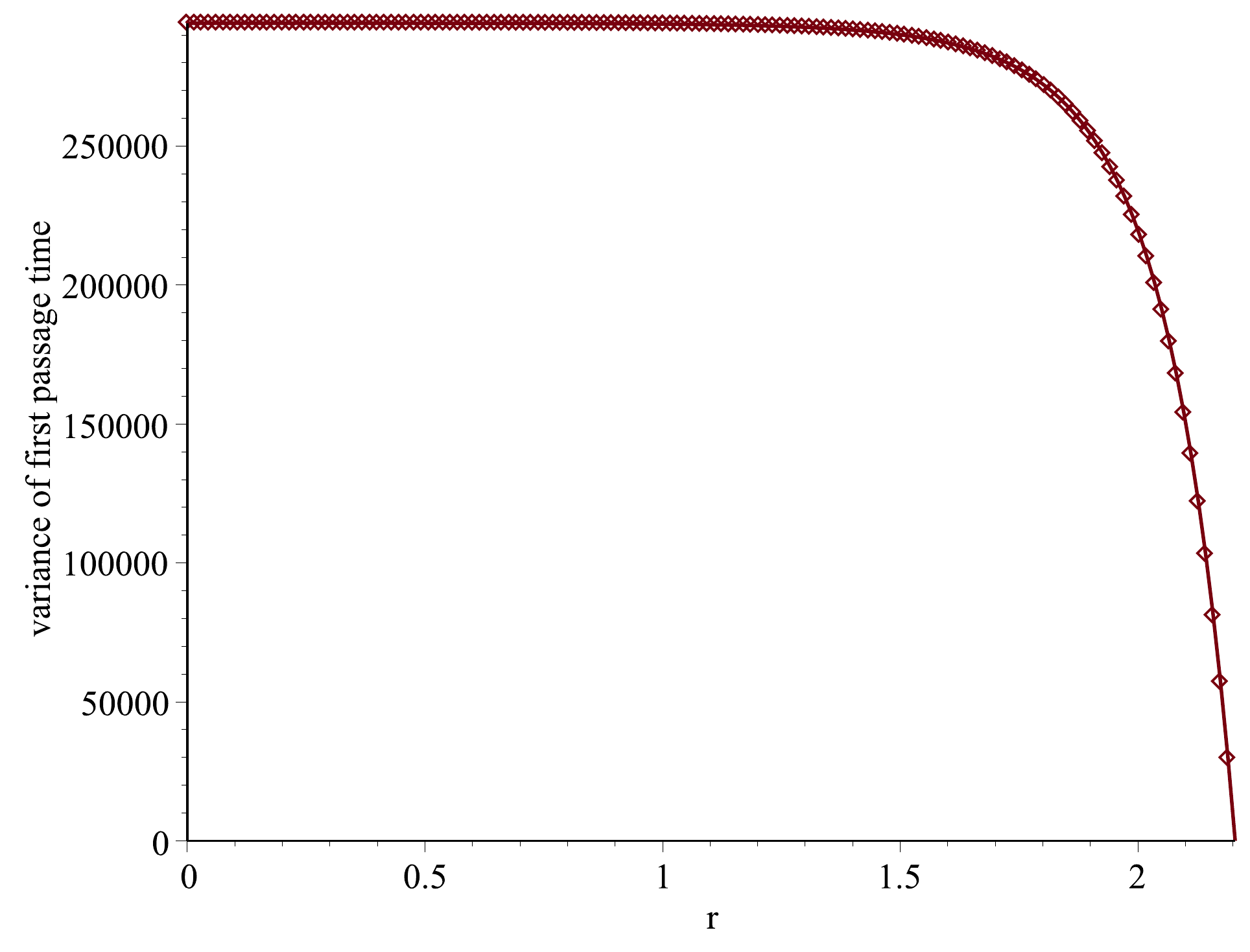}
\caption{Comparison of variance of first passage times. Solution from Theorem~2~(\textcolor{red}{$\diamond$}), solution from~\cite{ariaratnam:1973}~(\textcolor{red}{-}).}\label{fig:VarianceexitTimeWhiteComparison}
\end{center}
\end{figure}
The results show, that Theorems 1 and 2 provide exact moments for the first passage time of the oscillator \ref{eq:linear_oscillator_ariaratnam}.
\section{Forced and damped Mathieu oscillator}

The forced and damped Mathieu oscillator is an extension of the above linear oscillator by parametric forcing. It is a basic model for many physical problems. 
%This Mathieu oscillator includes parametric and external excitation without any aditional nonlinearities. 
The state space representation of the Mathieu oscillator with additive forcing $\xi_1(t)$ and parametric forcing $\xi_2(t)$ can be written as
\begin{equation}\label{eq:Mathieu}
\begin{aligned}
\frac{\mathrm{d}}{\mathrm{d} t} x&=y,\\
\frac{\mathrm{d}}{\mathrm{d} t}y&=-\alpha_1 \,x -\varepsilon\,\beta_1\, y  + \sqrt{\varepsilon}(\nu_1\,\xi_1(t) +  \nu_2\, x \,\xi_2(t)).
\end{aligned}
\end{equation}
Where $t$ is time,

Rewriting equation \eqref{eq:Mathieu} in terms of the total energy using the Hamilton function
\begin{equation}\label{eq:HamiltonianxyMathieu}
H(x,y):=\frac{y^2}{2}+\alpha_1\frac{x^2}{2},
\end{equation}
and its time derivative
\begin{equation}\label{eq:derivativeHamiltonianMathieu}
\begin{aligned}
\frac{\mathrm{d}}{\mathrm{d} t} H=-\varepsilon \,\beta_1\,y^2+\sqrt{\varepsilon}\, y\, (\nu_1\,\xi_1(t) +  \nu_2\, x\, \xi_2(t))
\end{aligned}
\end{equation}

Combining the first equation of \eqref{eq:Mathieu} with equation \eqref{eq:derivativeHamiltonianMathieu}  and rearrange equation \eqref{eq:HamiltonianxyMathieu} 
with respect to $y$ such that
\begin{equation}
\begin{aligned}
Q(x,H):=y^2=2H-\alpha  x^2, 
\end{aligned}
\end{equation}
we get the system
\begin{equation}\label{eq:averagingsystemMathieu}
\begin{aligned}
\frac{\mathrm{d}}{\mathrm{d} t} x&=\sqrt{Q(x,H)},\\
\frac{\mathrm{d}}{\mathrm{d} t} H&=-\varepsilon \,Q(x,H)\,\beta+\sqrt{\varepsilon}\, \sqrt{Q(x,H)} \,(\nu_1\,\xi_1(t) +  \nu_2\, x\,\xi_2(t)).
\end{aligned}
\end{equation}
An important property of equation \eqref{eq:averagingsystemMathieu} is,
that the energy level $H$ changes slowly compared to the
oscillations of the variable $x$. This enables the application of
stochastic averaging to this system.

\subsection{Mathieu oscillator excited by non-white Gaussian processes}\label{subsec:colloredNoiseLinear}
The solution of the unforced oscillator~\eqref{eq:Mathieu} with  $\varepsilon = 0$, energy level $H$, and initial conditions $x(0)=0$ and $y(0)=\sqrt{2\, H}$ is given by
\begin{equation}\label{eq:xsin}
\begin{aligned}
x(t)=b\,\sin(q\,t),\;\;\;\;\;\;\;x(t+s)=b\;\sin(q\,t +q\,s).
\end{aligned}
\end{equation}
\begin{equation}\label{eq:ycos}
\begin{aligned}
&y(t)=\sqrt{Q(x(t),H)}=b\,q\,\cos(q\,t),\\
&y(t+s)=\sqrt{Q(x(t+s),H)}=b\,q\,\cos(q\,t +q\,s).
\end{aligned}
\end{equation}
Here, $b=\sqrt{\frac{2\,H}{\alpha_1}}$ is the oscillation amplitude and $ q =\sqrt{\alpha_1} $ is the natural frequency of the linear oscillator. For every energy level $ H $ the oscillation period is given by
\begin{equation}
T=\frac{2\,\pi}{\sqrt{\alpha_1}}.
\end{equation}
Stochastic averaging of system \eqref{eq:Mathieu} by means of theorem~\ref{eq:satzDostalX} yields the convergence of the energy $ H $ for $\varepsilon \rightarrow
0$ to the solution of the one-dimensional It\^{o} equation
\begin{equation}\label{eq:onedimItolinear}
\mathrm{d}H=m(H)\,\mathrm{d}\tau+\sigma(H)\, \mathrm{d}W_{\tau},
\end{equation}
whereby $W_{\tau}$ is the standard Wiener process and ${\tau}=\varepsilon t$. 
The drift coefficient $m(H)$ in equation \eqref{eq:onedimItolinear} is given by
\begin{equation}\label{eq:averagedmLinear2}
\begin{aligned}
m(H)&=-\beta_1\, H+\frac{1}{T}\int_{-\infty}^0 \Big\{R_{\xi_1\xi_1}(s) \,\frac{\pi\,\nu_1^2}{q}\cos(q\,s)\\
&\quad+R_{\xi_2\xi_2}(s) \,\frac{\pi \,b^2\,\nu_2^2}{q}\left( \cos^2(q\,s)- \sin^2(q\,s) \right) \Big\}\,\mathrm{d}s.
\end{aligned}
\end{equation}
Using the trigonometric identity
\begin{equation}
\cos^2(q\,s)-\sin^2(q\,s)=\cos(2\,q\,s) \notag
\end{equation}
And the spectral density $S_{\xi_i\xi_i}(\omega)=\frac{1}{\pi}\int_{-\infty}^0 R_{\xi_i\xi_i}(s)\cos(\omega \,s)$, we finally obtain
\begin{equation}\label{eq:averagedmLinear3}
\begin{aligned}
m(H)=-\beta_1\, H+ \pi\,\nu_1^2\, S_{\xi_1\xi_1}(\sqrt{\alpha_1})+\frac{\pi\, \nu_2^2\, H}{\alpha_1} S_{\xi_2\xi_2}(2\,\sqrt{\alpha_1}) .
\end{aligned}
\end{equation}
For the diffusion coefficient $\sigma(H)$ of the equation \eqref{eq:onedimItolinear} we get
\begin{equation}\label{eq:averagedsigmaLinear}
\begin{aligned}
\sigma^2(H)&=\frac{b^2\,q^2}{T}\int_{-\infty}^\infty  \Big\{R_{\xi_1\xi_1}(s)\,\nu_1^2\int_{0}^{T}\cos(q\,t+q\,s)\cos(q\,t)\,\mathrm{d}t\\
&\quad+R_{\xi_2\xi_2}(s)\,b^2\,\nu_2^2\int_{0}^{T}\sin(q\,t+q\,s)\sin(q\,t)\cos(q\,t+q\,s)\cos(q\,t)\,\mathrm{d}t\Big\} \,\mathrm{d}s.
\end{aligned}
\end{equation}
After evaluation of these integrals we obtain
\begin{equation}\label{eq:averagedsigmaLinear3}
\sigma^2 (H) = 2\,\pi \, \nu_1^2\, S_{\xi_1\xi_1}(\sqrt{\alpha_1})\, H+\frac{ \pi \,\nu_2^2 }{\alpha_1}\, S_{\xi_2\xi_2}(2\,\sqrt{\alpha_1})\,H^2.
\end{equation}

\subsection{Validation of the results for the Mathieu oscillator}
The results of the stochastic averaging of the Mathieu oscillator according to theorem \ref{eq:satzDostalX} are validated using the well-known solution of the classical stochastic averaging of the stochastic linear oscillator is used. This was determined by Ariaratnam and Tam in \cite{ariaratnam:1976}. The drift and diffusion coefficients $m_{A}$ and $\sigma_{A}$ are  functions of the oscillation amplitude $ b $.
\begin{equation}\label{eq:averagedmLinearAriaratnam}
m_{A}(b)=-\frac{\beta_1}{2} \,b+ \frac{\pi}{2\,\alpha_1}\, \nu_1^2\, S_{\xi_1\xi_1}(\sqrt{\alpha_1})\,\frac{1}{b}+\frac{3\,\pi }{8\,\alpha_1}\, \nu_2^2 \,S_{\xi_2\xi_2}(2\,\sqrt{\alpha_1})\,b ,
\end{equation}
\begin{equation}\label{eq:averagedsigmaLinearAriaratnam}
\sigma_{A}^2 (b) =\frac{ \pi}{\alpha_1}\,\nu_1^2\, S_{\xi_1\xi_1}(\sqrt{\alpha_1}) +\frac{ \pi }{4\,\alpha_1}\, \nu_2^2 \,S_{\xi_2\xi_2}(2\,\sqrt{\alpha_1})\,b^2
\end{equation}

Since the drift and diffusion coefficients according to \cite{ariaratnam:1976} are not functions of the energy $ H $ but functions of the oscillation amplitude $ b $, the probability densities of the oscillation amplitude $ b $ are calculated in order to compare the results. The stationary solution \eqref{eq:statDensity} of the Fokker-Planck equation \eqref{eq:FPKeq} using the drift and diffusion coefficients from \eqref{eq:averagedmLinearAriaratnam} and \eqref{eq:averagedsigmaLinearAriaratnam} is given by
\begin{equation}\label{eq:statDensityAriaratnam}
p_{A}(b)=\frac{\tilde{C}}{\sigma_{A}^2(b)}\exp\!\left(2\int_0^b\frac{m_{A}(\zeta)}{\sigma_{A}^2(\zeta)}\,\mathrm{d}\zeta \right),
\end{equation}
where $ \tilde{C} $ is a normalization constant.
For the stationary probability density of the stochastic linear oscillator using the drift and diffusion coefficients \eqref{eq:averagedmLinear3} and \eqref{eq:averagedsigmaLinear3} from the stochastic averaging of the energy by theorem \ref{eq:satzDostalX} we obtain
\begin{equation}\label{eq:statDensityLinear3}
p_{\mathrm{Lin}}(H)=\frac{C}{\sigma^2(H)}\exp\!\left(2\int_0^H\frac{m(\zeta)}{\sigma^2(\zeta)}\,\mathrm{d}\zeta \right),
\end{equation}
with a normalization constant $C$.

Using the transformation
%Mit der Transformation
$p_{\mathrm{st}}(b)=p_{\mathrm{st}}(H)\left(\frac{\mathrm{d}}{\mathrm{d}H}b\right)^{-1}$, the probability density $p_{\mathrm{Lin}}(H)$ can be expressed a s a function of the oscillation amplitude $b$ by

\begin{equation}\label{eq:statDensitybLinear}
p_{\mathrm{Lin}}(b)=p_{\mathrm{Lin}}(H)\, \sqrt{2\,\alpha_1 H}.
\end{equation}

As can be seen from Figure~\ref{fig:linearMittelungsmethodenVergleich}, the two probability density functions $p_{\mathrm{Lin}}(b)$ und $p_{A}(b)$ are identical for the chosen parameters.
Thus, it is obvious that the results of the two stochastic averaging methods for the stochastic linear oscillator coincide. The correspondence of these averaging methods can generally be shown by transforming the stochastic differential equation for the energy $ H $ \eqref{eq:onedimItolinear} into an equivalent stochastic differential equation in terms of the oscillation amplitude $ b $.
\begin{figure}[htbp]
	\begin{center}
		\includegraphics[width=12.0cm]{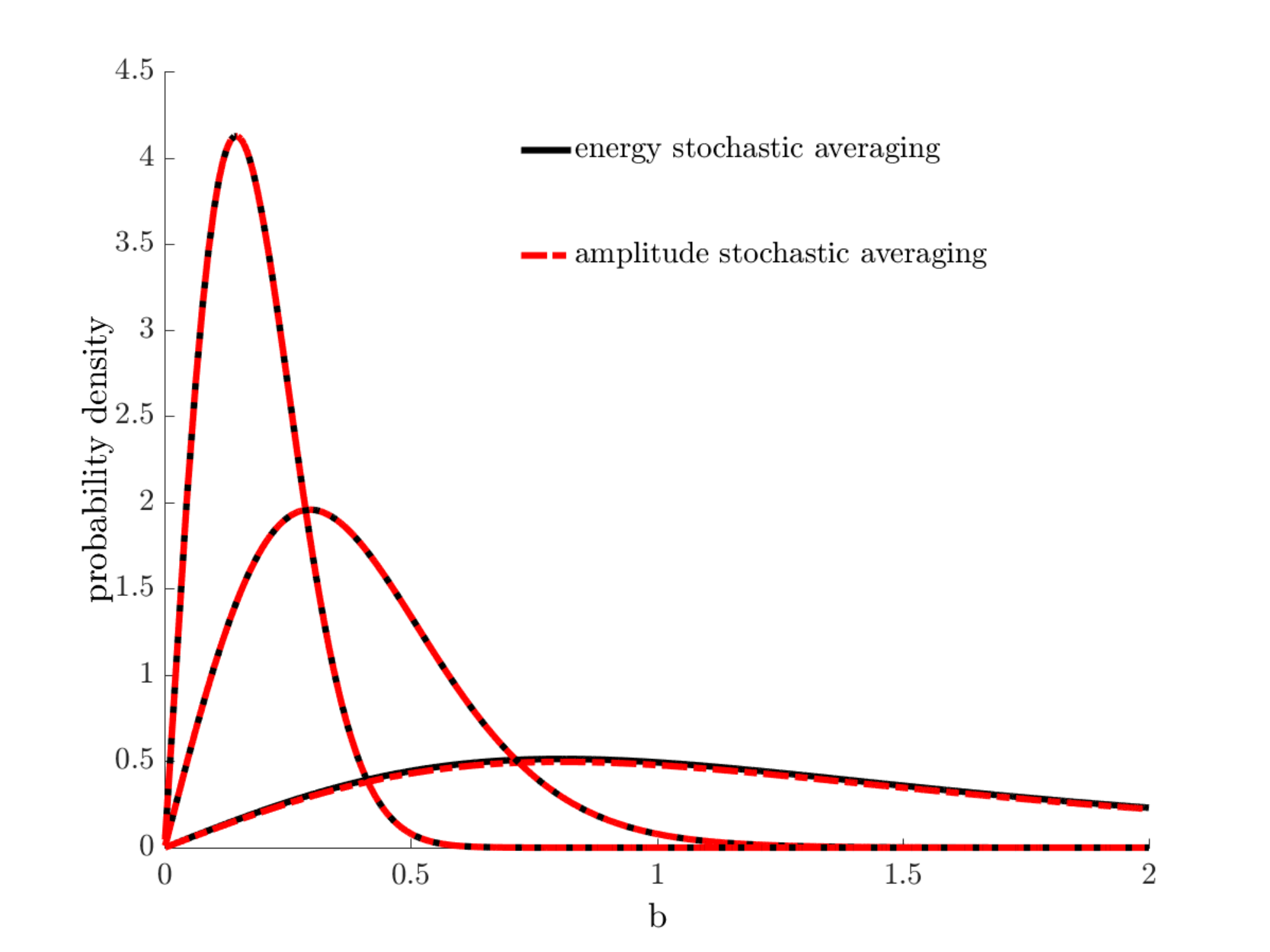}
		\caption{Comparison of energy stochastic averaging according to Theorem~\ref{eq:satzDostalX} (\protect\legendlineb{black}{solid}{2pt}) with amplitude stochastic averaging according to \cite{ariaratnam:1976} (\protect\legendlineb{red}{dashed}{2pt}\!) for the probability density of the oscillation amplitude $ b $ and various values of the damping parameter $\beta_1 \in \{0,03;0,18;0,72\}$.}\label{fig:linearMittelungsmethodenVergleich}
	\end{center}
\end{figure}

%\subsection{First passage times for the forced and damped Mathieu oscillator}\label{sec:first_passage_Mathieu}
%
%Results for the moments of first passage times of  the  forced and damped Mathieu oscillator~\eqref{eq:Mathieu} can be determined according Theorems~1 and 2 using the analytical expressions for the drift and diffusion in equations~\eqref{eq:averagedmLinear3} and \eqref{eq:averagedsigmaLinear3}, respectively.
%Then, the expressions for the first two moments $M_1(H_0,H_c)$ and $M_2(H_0,H_c)$ of the first passage times for the forced and damped Mathieu oscillator~\eqref{eq:Mathieu} until the energy process $H(t)$ reaches the value $0< H_c $ starting at $H_0\in(0,H_c)$ are given by
%...
%
%
%With this the corresponding mean $M(r_0,R)=M_1(r_0,R)$ and variance $V(r_0,R)=M_2(r_0,R)-(M_1(r_0,R))^2$ of the considered first passage time are 
%...

\section{Nonlinear system with external and parametric excitation}\label{sec:exampleParametric}
In this section a Duffing oscillator with softening cubic stiffness and linear, quadratic and cubic damping is investigated.
Such a Duffing oscillator with the scalar state variables $ x (t) $ and $ y (t) $ is given by 
\begin{equation}\label{eq:rollequation3a}
\begin{aligned}
\frac{\mathrm{d}}{\mathrm{d} t} x&=y,\\
\frac{\mathrm{d}}{\mathrm{d} t}y&=-\alpha_1 \,x + \alpha_3\, x^3 -\varepsilon\,(\beta_1\, y+\beta_2 \,|y|\,y+\beta_3\, y^3)  + \sqrt{\varepsilon}\,(\nu_1\,\xi_1(t) +  \nu_2\, x \,\xi_2(t)).
\end{aligned}
\end{equation}
We use the
Hamiltonian formalism and rewrite equation \eqref{eq:rollequation3a}
in terms of the total energy defined by the Hamilton function
\begin{equation}\label{eq:Hamiltonianxy}
H(x,y):=\frac{y^2}{2}+\alpha_1\frac{x^2}{2}-\alpha_3\frac{x^4}{4},
\end{equation}
where $\alpha_1,\alpha_3>0$. Then \eqref{eq:rollequation3a} can be
written as
\begin{equation}\label{eq:rollequation3}
\begin{aligned}
\frac{\mathrm{d}}{\mathrm{d} t}x&=\frac{\partial H(x,y)}{\partial y},\\
\frac{\mathrm{d}}{\mathrm{d} t}y&= -\frac{\partial H(x,y)}{\partial x}-\varepsilon\,\frac{\partial H}{\partial y} (\beta_1+\beta_2 \,|y|+\beta_3\, y^2) +\sqrt{\varepsilon}\, (\nu_1 \,\xi_1(t) +  \nu_2 \,x\, \xi_2(t)).
\end{aligned}
\end{equation}
Thereby we assume $\varepsilon<<1$ to be a small parameter. With this approach we have
transformed the original equation \eqref{eq:rollequation3a} to a weakly perturbed Hamiltonian
system \eqref{eq:rollequation3} excited by the stationary random
processes $\xi_1(t)$ and $\xi_2(t)$. The coefficients $\nu_1$ and
$\nu_2$ are additive and parametric noise intensities, respectively.
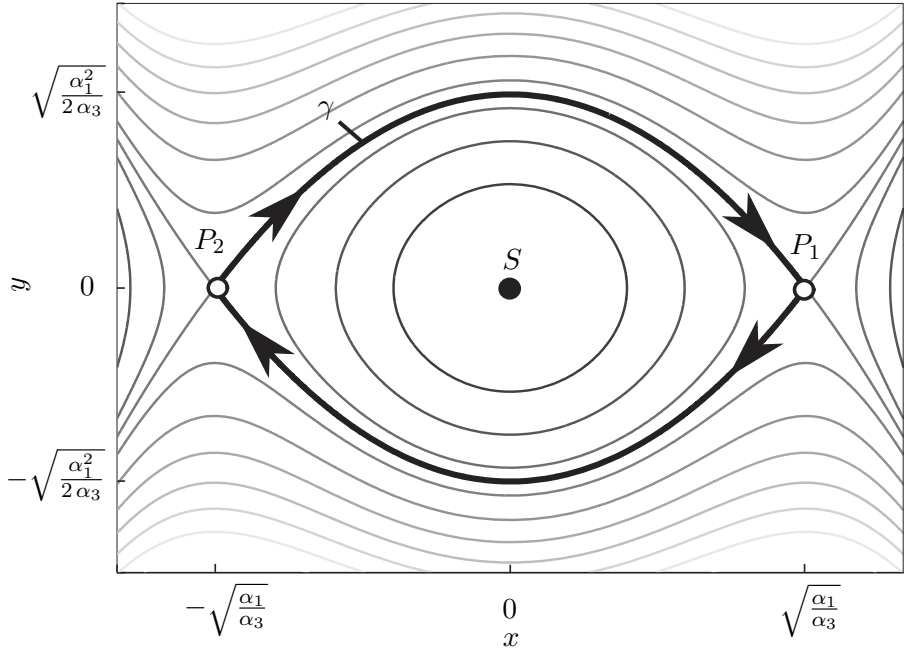
\begin{figure}[htbp]
		\begin{center}
		%\input{figHamiltonianCubic2IncKr3.eps_tex}%%%%%%%%%%%%%%%%%%%%%%%%
		
	%% Creator: Inkscape 0.48.2, www.inkscape.org
	%% Accompanies image file 'figHamiltonianCubic2IncKr3.eps' (pdf, eps, ps)
	\begingroup%
	\makeatletter%
	\providecommand\color[2][]{%
		\errmessage{(Inkscape) Color is used for the text in Inkscape, but the package 'color.sty' is not loaded}%
		\renewcommand\color[2][]{}%
	}%
	\providecommand\transparent[1]{%
		\errmessage{(Inkscape) Transparency is used (non-zero) for the text in Inkscape, but the package 'transparent.sty' is not loaded}%
		\renewcommand\transparent[1]{}%
	}%
	\providecommand\rotatebox[2]{#2}%
	\ifx\svgwidth\undefined%
	\setlength{\unitlength}{332.10134711bp}%
	\ifx\svgscale\undefined%
	\relax%
	\else%
	\setlength{\unitlength}{\unitlength * \real{\svgscale}}%
	\fi%
	\else%
	\setlength{\unitlength}{\svgwidth}%
	\fi%
	\global\let\svgwidth\undefined%
	\global\let\svgscale\undefined%
	\makeatother%
	\begin{picture}(1,0.72224477)%
	\put(0,0){\includegraphics[width=\unitlength]{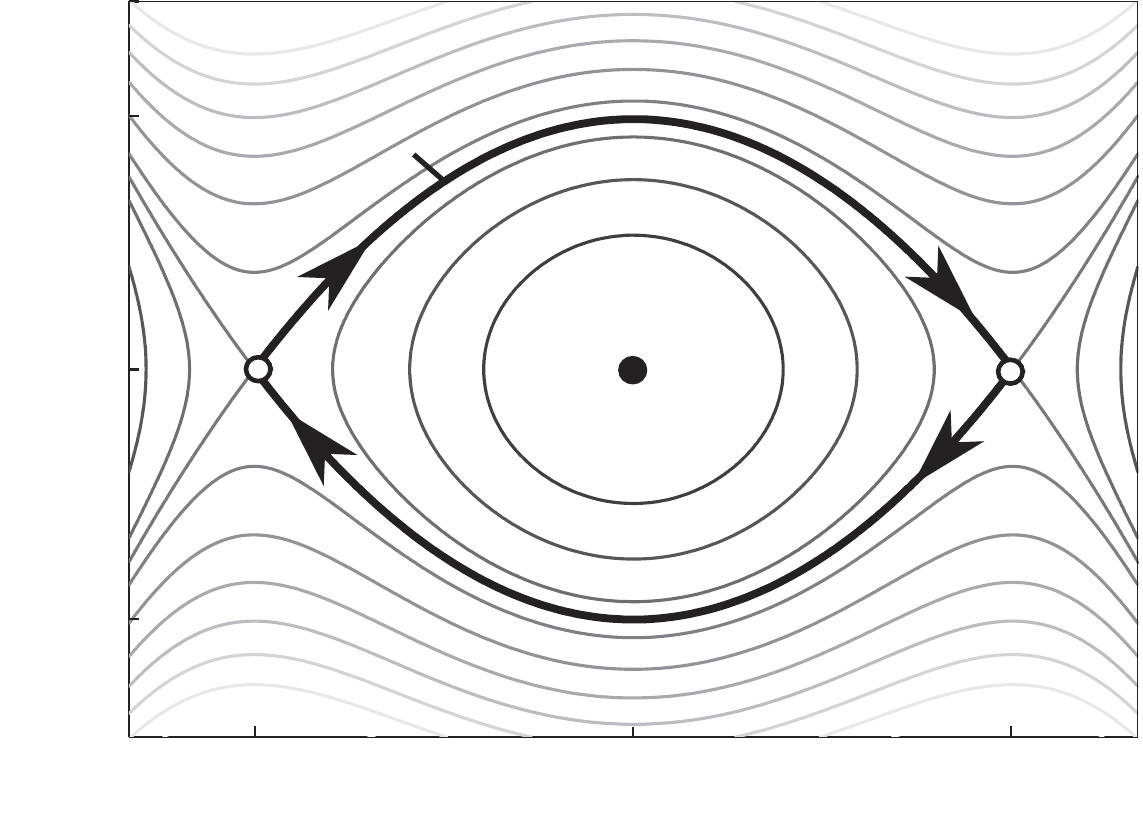}}%
	\put(0.05146644,0.40419841){\color[rgb]{0,0,0}\rotatebox{90}{\makebox(0,0)[lb]{\vspace{5mm}\smash{$y$}}}}%
	\put(0.07096555,0.40155227){\color[rgb]{0,0,0}\rotatebox{-1.78181179}{\makebox(0,0)[lb]{\smash{$0$}}}}%
	\put(0.19956667,0.45330201){\color[rgb]{0,0,0}\makebox(0,0)[lb]{\smash{$P_2$}}}%
	\put(0.33835309,0.60651231){\color[rgb]{0,0,0}\makebox(0,0)[lb]{\smash{$\gamma$}}}%
	\put(0.54584616,0.43046476){\color[rgb]{0,0,0}\makebox(0,0)[lb]{\smash{$S$}}}%
	\put(0.8686793,0.44712826){\color[rgb]{0,0,0}\makebox(0,0)[lb]{\smash{$ P_1$}}}%
	\put(0.54565181,0.03585608){\color[rgb]{0,0,0}\makebox(0,0)[lb]{\smash{$0$}}}%
	\put(0.54599995,0.00362451){\color[rgb]{0,0,0}\makebox(0,0)[lb]{\smash{$x$}}}%
	\put(0.85666832,0.03764191){\color[rgb]{0,0,0}\makebox(0,0)[lb]{\smash{$\sqrt{\frac{\alpha_1}{\alpha_3}}$}}}%
	\put(0.18819035,0.03872686){\color[rgb]{0,0,0}\makebox(0,0)[lb]{\smash{$-\sqrt{\frac{\alpha_1}{\alpha_3}}$}}}%
	\put(-0.00098449,0.18401901){\color[rgb]{0,0,0}\makebox(0,0)[lb]{\smash{$\!\!-\sqrt{\frac{\alpha_1^2}{2\,\alpha_3}}$}}}%
	\put(0.02536807,0.61846668){\color[rgb]{0,0,0}\makebox(0,0)[lb]{\smash{$\!\!\sqrt{\frac{\alpha_1^2}{2\,\alpha_3}}$}}}%
	\end{picture}%
	\endgroup%

		%%%%%%%%%%%%%%%%%%%%%%%%%%%%%%%%%%%%%%%%%%%%%%%%%%%%%%%%%%%%%%%%%%%%
	\caption{Contour lines of $H(x,y)$.}\label{fig:HamiltonianShipBar2}
		\end{center}
\end{figure}
The fixed points of system \eqref{eq:rollequation3} without
dissipation and random perturbation, i.e. $\varepsilon=0$, are
\begin{equation}\label{eq:hamiltonianfixedpoints}
P_1=\left(\sqrt{\frac{\alpha_1}{\alpha_3}}, 0\right);\;\;P_2=\left(-\sqrt{\frac{\alpha_1}{\alpha_3}}, 0\right);\;\;S=(0,0).
\end{equation}
These saddle points $P_1$ and $P_2$ are connected by the
heteroclinic orbit
\begin{equation}\label{eq:heteroclinicOrbit}
\gamma(x,y)=\left\{x,y\in \mathds{R},\, |x|<\sqrt{\frac{\alpha_1}{\alpha_3}} :\,y^2+\alpha_1\, x^2-\frac{\alpha_3}{2}\,x^4=\frac{\alpha_1^2}{2\,\alpha_3}\right\}
\end{equation}

The considered duffing oscillator oscillates only within the region bounded by the heteroclinic orbit $ \gamma $. This region is given by 
\begin{equation}\label{eq:heteroclinicOrbitInneres}
\mathbf{D}_{\gamma}:=\left\{x,y\in \mathds{R},\, |x|<\sqrt{\frac{\alpha_1}{\alpha_3}} :\,y^2+\alpha_1 \,x^2-\frac{\alpha_3}{2}\,x^4<\frac{\alpha_1^2}{2\,\alpha_3}\right\}.
\end{equation}%\sqrt{\frac{\alpha_1^2}{2\alpha_3}}
If the conservative system given by \eqref{eq:rollequation3a} with $\varepsilon=0$ is considered, then for every energy level $H(x,y)$ with $(x,y)\in \mathbf{D}_{\gamma}$ exactly one closed trajectory exists in the phase space of the conservative system.
These trajectories correspond to the contour lines of the Hamiltonian \eqref{eq:Hamiltonianxy} and are
shown in figure \ref{fig:HamiltonianShipBar2}.
The time derivative of Hamiltonian for system
\eqref{eq:rollequation3} is given by
\begin{equation}\label{eq:derivativeHamiltonian2}
\begin{aligned}
\frac{\mathrm{d}}{\mathrm{d} t} H=\varepsilon \,y^2(-\beta_1-\beta_2\,|y|-\beta_3\, y^2)+\sqrt{\varepsilon}\, y\, (\nu_1\,\xi_1(t) +  \nu_2\, x\, \xi_2(t))
\end{aligned}
\end{equation}
Combining the first equation of \eqref{eq:rollequation3} and
equation \eqref{eq:derivativeHamiltonian2}, and using
\begin{equation}
\begin{aligned}
Q(x,H):=y^2=2H-\alpha_1 x^2+\alpha_3\frac{x^4}{2}
\end{aligned}
\end{equation}
obtained from \eqref{eq:Hamiltonianxy}, we get the system
\begin{equation}\label{eq:averagingsystem1}
\begin{aligned}
\frac{\mathrm{d}}{\mathrm{d} t} x&=\sqrt{Q(x,H)},\\
\frac{\mathrm{d}}{\mathrm{d} t} H&=\varepsilon \,Q(x,H)\,(-\beta_1-\beta_2 \,\sqrt{Q(x,H)}-\beta_3 \,Q(x,H))\\
&\hspace{3.0cm}+\sqrt{\varepsilon}\, \sqrt{Q(x,H)} \,(\nu_1\,\xi_1(t) +  \nu_2\, x\,\xi_2(t)).
\end{aligned}
\end{equation}

Since as before  a property of  equation \eqref{eq:averagingsystem1} is that the energy level $H$ changes slowly compared to the oscillations of the variable $x$, stochastic averaging can be applied.

%%%%%%%%%%%%%%%%%%%%%%%%%%%%%%%%%%%%%%%
\subsection{White noise case}\label{subsec:whiteNoise}
First we state the results for averaging system
\eqref{eq:rollequation3} subjected to white noise excitation.
Therefore, let $\int_0^t \xi_{s} ds=W_t$, where $W_t$ is a
standard Wiener process and $dW_t$ its increment.
Then, using the It\^{o} lemma, we obtain from system
\eqref{eq:rollequation3}
\begin{equation}\label{eq:systemforaverageWhiteNoise}
\begin{aligned}
dx&=\sqrt{Q(x,H)}dt,\\
dH&=\varepsilon \{Q(x,H)(-\beta_1-\beta_3 Q(x,H))+F(x)\}dt+\sqrt{\varepsilon} \sqrt{Q(x,H)} (\nu_1 + x \nu_2) dW_t.
\end{aligned}
\end{equation}
Averaging \eqref{eq:systemforaverageWhiteNoise} according to
\cite{Khasminskii:1968} we obtain the one dimensional It\^{o}
equation
\begin{equation}\label{eq:onedimIto}
dH=m(H)d\tau+\sigma^2(H) dW_{\tau},
\end{equation}
where
\begin{equation}\label{eq:averagedHm}
\begin{aligned}
m(H)=\frac{1}{T(H)}\int_{0}^{T(H)}\left\{Q(x(t),H)G(x(t),H)+F(x(t))\right\}dt,
\end{aligned}
\end{equation}
\begin{equation}\label{eq:averagedHsigma}
\begin{aligned}
\sigma^2(H)=\frac{2}{T(H)}\int_{0}^{T(H)}\left\{Q(x(t),H)F(x(t))\right\}dt.
\end{aligned}
\end{equation}
Here, $T(H)$ is the period of one oscillation of the fast variable
$x$ in the absence of noise and damping, i.e. $\varepsilon=0$,
starting at the energy level $H$. Additionally we use
\begin{equation}
G(x,H)=-\beta_1 -\beta_3 Q(x(t),H),\;\;\;\;\;\;F(x)=\frac{1}{2}\left(\nu_1^2+\nu_2^2x^2\right).
\end{equation}
For $0 \leq H < \alpha_1^2/(4\alpha_3)$ the period is given by
\begin{equation}\label{eq:periodTH}
T(H)=\int_0^{T(H)}dt=2\int_{-b(H)}^{b(H)}\frac{dx}{\sqrt{Q(x,H)}}=\frac{4}{q}K(k),
\end{equation}
where
\begin{equation}
\begin{aligned}
q=a \sqrt{\frac{\alpha_3}{2}},
\end{aligned}
\end{equation}
\begin{equation}\label{eq:a}
\begin{aligned}
a=\sqrt{\frac{4H}{b^2\alpha_3}}.
\end{aligned}
\end{equation}
The limits of integration $\pm b(H)$ are the points where
$y=\sqrt{Q(x,H)}=0$ and the periodic orbit intersects the x-axis,
i.e. $b(H)$  is the maximum value of $x$ for each energy level $H$
and is given by
\begin{equation}\label{eq:xplus}
\begin{aligned}
b=&\sqrt{-\frac{-\alpha_1+\sqrt{\alpha_1^2-4\alpha_3H}}{\alpha_3}}.
\end{aligned}
\end{equation}
The function $K(k)$ is the complete elliptic integral of the first
kind, cf. \cite{byrd:1954}. The elliptic modulus $k$ is given by
\begin{equation}
\begin{aligned}
k=\frac{b}{a}.
\end{aligned}
\end{equation}
%%%%%%%%%%%%%%%%%%%%%%%%%%%%%%%%%%%%%%%%%%%%%%%%%%%%%%%%%%%
%
The integrals appearing in the equations for drift
\eqref{eq:averagedHm} and diffusion \eqref{eq:averagedHsigma} exist
for $0 \leq H < \alpha_1^2/(4\alpha_3)$. They can be computed in
terms of complete elliptic integrals of the first and second kind,
$K(k)$ and $E(k)$, respectively. Then we get
\begin{equation}
\begin{aligned}
m(H)=B(H)+C(H),
\end{aligned}
\end{equation}
\begin{equation}
\begin{aligned}
\sigma^2(H)=\nu_1^2 B_1(H) + \nu_2^2 B_2(H),
\end{aligned}
\end{equation}
\begin{equation}
\begin{aligned}
B(H)=- \left( \beta_{{1}}+2\,\beta_{{3}}H \right)
B_{{1}}(H)+\alpha_{{1}}\beta
_{{3}}B_{{2}}(H)-\frac{1}{2}\,\alpha_{{3}}\beta_{{3}}B_{{3}}(H),
\end{aligned}
\end{equation}
\begin{equation}
\begin{aligned}
B_1(H)=\frac{1}{3}\,{q}^{2} \left[ {b}^{2}-{a}^{2} +\left( {a}^{2}+{b}^{2} \right) \frac{E\left( k \right)}{K\left( k \right)} \right],
\end{aligned}
\end{equation}
\begin{equation}
\begin{aligned}
B_2(H)=\frac{1}{15}\,{q}^{2} \left[3\,{a}^{2}{b}^{2}-2\,{a}^{4}-{b}^{4}  +\left( 2\,{a}^{4}+2\,{b}^{4}-2\,{a}^{2}{b}
^{2} \right) \frac{E\left( k \right)}{K\left( k \right)} \right],
\end{aligned}
\end{equation}
\begin{equation}
\begin{aligned}
B_3(H)={\frac {1}{105}}\,{q}^{2} \Bigg[ &3\,{b}^{4}{a}^{2}+9\,{b}^{2}{a}^{4}-4\,{b}^{6}-8\,{a}^{6}\\
&+\left(8\,{a}^{6}+8\,{b}^{6}-5\,{b}^{2}{a}^{4}-5\,{b}^{4}{a}^{2} \right) \frac{E\left( k \right)}{K\left( k \right)} \Bigg],
\end{aligned}
\end{equation}
\begin{equation}
\begin{aligned}
C(H)=\frac{1}{2}\,{\nu_{{1}}}^{2}-\frac{1}{2}\,{\nu_{{2}}}^{2}{a}^{2}
\left({\frac{E\left( k \right)}{K\left( k \right)}}-1\right).
\end{aligned}
\end{equation}
Thus we have obtained an one-dimensional It\^{o} equation
\eqref{eq:onedimIto} for the process of total energy $H(t)$ of
system \eqref{eq:rollequation3} subjected to white noise excitation.
%
%%%%%%%%%%%%%%%%%%%%%%%%%%%%%%%%%%%%%%%%%%%%%%%%%%%%%
%
\subsection{Real noise case}
For the non-white noise case it is necessary to determine the
functions $x(t+s)$ and $y(t+s)$ in terms of $x(t)$ and $y(t)$
in order to apply stochastic averaging and obtain a closed form
solution. Therefore, a solution of the differential
equation \eqref{eq:rollequation3} is needed, which can be obtained
for $\varepsilon=0$ in terms of Jacobian elliptic functions. Then
addition formulas for Jacobian elliptic functions can be used to
eliminate the time shift and obtain the states $x$ and $y$ as
functions of time $t$ only. In the case $0\leq
H<\alpha_1^2/(4\alpha_3)$  we get for
$\varepsilon=0$ a solution of \eqref{eq:rollequation3} by
\begin{equation}\label{eq:xsn}
\begin{aligned}
x(t)=b\;\sn(qt,k),\;\;\;\;\;\;\;x(t+s)=b\;\sn(qt +qs,k),
\end{aligned}
\end{equation}
\begin{equation}\label{eq:ycndn}
\begin{aligned}
&\frac{dx}{dt}=y(t)=\sqrt{Q(x(t),H)}=b\:q\;\cn(qt,k)\:\dn(qt,k),\\
&y(t+s)=\sqrt{Q(x(t+s),H)}=b\:q\;\cn(qt +qs,k)\:\dn(qt +qs,k),
\end{aligned}
\end{equation}
with
$b\:q=\sqrt{2H}$. The expressions for $x(\cdot)$ and $y(\cdot)$
contain the Jacobian elliptic functions $\sn(\cdot,k)$, $\cn(\cdot,k)$
and $\dn(\cdot,k)$, see \cite{byrd:1954}. The elliptic modulus is the
same as in the white noise case. We use the abbreviations $\sn:=\sn(qt,k), \;\cn:=\cn(qt,k), \; \dn:=\dn(qt,k),  \;u:=qt.$
In addition if the subscript $s$ or $t+s$ is used, we refer to the
argument $qs$ or $q(t+s)$, respectively.
%%%%%%%%%%%%%%%%%%%%%%%%%%%%%%%%%%%%%%%%%%%%%%%%%%%%%%%%%%%%%%%%%%%%%%%%%%%%%%
%
Applying stochastic averaging according to Theorem \ref{eq:satzDostalX} to system \eqref{eq:averagingsystem1} for $0\leq H <\frac{\alpha_1^2}{4 \alpha_3}$,
we get the one-dimensional It\^{o} stochastic differential equation
	\begin{equation}\label{eq:averageIto}
	\mathrm{d}H=m(H)\,\mathrm{d}\tau+\sigma(H)\, \mathrm{d}W_{\tau}.
	\end{equation}
for the energy level $H$, where $\mathrm{d}W_t$ is a standard Wiener process.
The corresponding drift $m(H)$ and diffusion $\sigma(H)$ in equation \eqref{eq:averageIto} are given by
\begin{equation}\label{eq:averagedm}
\begin{aligned}
m(H)&=\frac{2}{T\,q}\int_{-\infty}^0 \Big\{R_{\xi_1\xi_1}(s) \,\nu_1^2\int_{-K(k)}^{K(k)}
\frac{\cn_{t+s}\,\dn_{t+s}}{\cn_{t}\,\dn_{t}} \,\mathrm{d}u\\
&\hspace{1.8cm}+R_{\xi_2\xi_2}(s) \,b^2\,\nu_2^2\int_{-K(k)}^{K(k)}\sn\,\sn_{t+s}\,\frac{\cn_{t+s}\,\dn_{t+s}}{\cn_{t}\,\dn_{t}} \,\mathrm{d}u\Big\}\,\mathrm{d}s\\
&\hspace{2.8cm}+ \frac{1}{T}\int_{0}^{T}Q(x(t),H)\,G(x(t),H)\,\mathrm{d}t
\end{aligned}
\end{equation}
\begin{equation}\label{eq:averagedsigma}
\begin{aligned}
\sigma^2(H)&=\frac{4\,b^2\,q}{T}\int_{-\infty}^\infty  \Big\{R_{\xi_1\xi_1}(s)\,\nu_1^2\int_{0}^{K(k)}\hspace{-7mm} \cn\,\dn\, \cn_{t+s}\,\dn_{t+s}\,\mathrm{d}u\\
&\hspace{2.0cm}+R_{\xi_2\xi_2}(s)\,b^2\,\nu_2^2\int_{0}^{K(k)}\hspace{-7mm} \sn\,\sn_{t+s}\, \cn\,\dn\, \cn_{t+s}\,\dn_{t+s}\,\mathrm{d}u \Big\} \,\mathrm{d}s
\end{aligned}
\end{equation}
Thereby we assume, that the involved stochastic processes $\xi_1(t)$ and $\xi_2(t)$ are stationary with mean zero and have
autocorrelation function
$R_{\xi_t\xi_t}(s)=E\{\xi_t\xi_{t+s}\}$, which approaches zero
sufficiently fast as $s$ increases.

\subsection{Results for first passage times}
With the theory proposed in this work, the first 
moment of the first passage time of the Duffing oscillator with negative cubic stiffness and nonlinear damping, as given in equation \eqref{eq:rollequation3a}, is determined. 
First, the It\^{o} equation~\eqref{eq:averageIto} for the averaged energy $H$ of this Duffing oscillator with equation~\eqref{eq:averagedm} for the drift $m$ and equation \eqref{eq:averagedsigma} for the diffusion $\sigma$ is used, which is a one dimensional diffusion process.
Then, Theorem~1 
%and 2 are 
is applied in order to calculate the mean %first
% two moment 
$M(H_0,H_c)=M_1(H_0,H_c)$ 
%and $M_2(H_0,H_c)$ 
of the first passage time until the energy process $H(t)$ reaches the value $0< H_c <\frac{\alpha_1^2}{4 \alpha_3}$ starting at $H_0\in(0,H_c)$. %From these results 
%Then, the mean of the considered first passage time is $M(H_0,H_c)=M_1(r_0,R)$ and variance $V(r_0,R)=M_2(r_0,R)-(M_1(r_0,R))^2$ of the considered first passage time are obtained.
%
Using the formula~\eqref{eq:xplus}, the energy can be transformed to the oscillation amplitude $b(t)$ of the Duffing oscillator.
The chosen parameter values for the softening Duffing oscillator are summarized in Table~\ref{tab:CompData}.
These values correspond to a 200 m long ship, as described in \cite{dostal:2012}, which is traveling with a velocity of 8 knots in a sea state with wave encounter angle $\chi=30^{\circ}$ a significant wave height $H_s=5$ m and mean wave period of $T_1=11.5$ s. Thereby the sea state has a JONSWAP spectral density, cf. \cite{dostal:2012}. %Then 
The resulting oscillation amplitude $b$ is the roll angle of the ship in degrees. The critical Energy $H_c=0.529$ is chosen, which corresponds to the roll oscillation amplitude $b_c=40$ degrees. 
Using the parameter values from  Table~\ref{tab:CompData}, the functions for drift $m(b)$ and diffusion $\sigma(b)$ are shown in Figures~\ref{fig:exitTimeDrift} and \ref{fig:exitTimeDiffusion}, respectively.
%alpha1*b_end^2/2-alpha3*b_end^4/4
The results for the mean first passage time of the softening Duffing oscillator, for reaching the boundary $b_c=40$ degrees starting at $b_0$ are shown in Figure~\ref{fig:exitTimeDuffingComparison} using Theorem 1 and Corollary 1.
The computation time of the corresponding formula \eqref{eq:meanexittime} from Theorem 1 using standard numerical quadrature is about 15 seconds on an ordinary desktop computer.
The mean first passage time obtained from Corollary 1 is only correct, if the starting oscillation amplitude $b_0$ is very close to zero. For higher values of the starting oscillation amplitude $b_0$, the mean first passage times have to be determined using Theorem~1.
%%%
%The corresponding variance of the first passage time obtained from Theorem 2 is shown in Figure \ref{ }.
%%%

 \begin{table}
 	\begin{center}
 		\caption{Parameter values used for computations}\label{tab:CompData}
 		\begin{tabular}[h]{l l l l l}
 			%   &&& \\
 			\hline  %\vspace{5mm}
 			$\alpha_1=3.187$ & $\alpha_3=4.164$ & $\beta_1 = 0.655$& $\beta_2 =0.921$ & $\beta_3=0$ \\
 			  $\nu_1= 0.018$ & $\nu_2=1.783$& $H_c=0.529$ &$b_c=40^{\circ}$& $\varepsilon=0.1$\\ %alpha1 =3.186961113216903 alpha3 = 4.163739106087820 beta1 = 0.654728473862060 beta2 =  0.920839914233673 beta3 =  0 nu2 = 1.783185015824232	 nu1 = 0.017944806063564		
 			\hline
 		\end{tabular}
 	\end{center}
 \end{table}   
   
 \begin{figure}[htbp]
 	\begin{center}
 		\includegraphics[width=10cm]{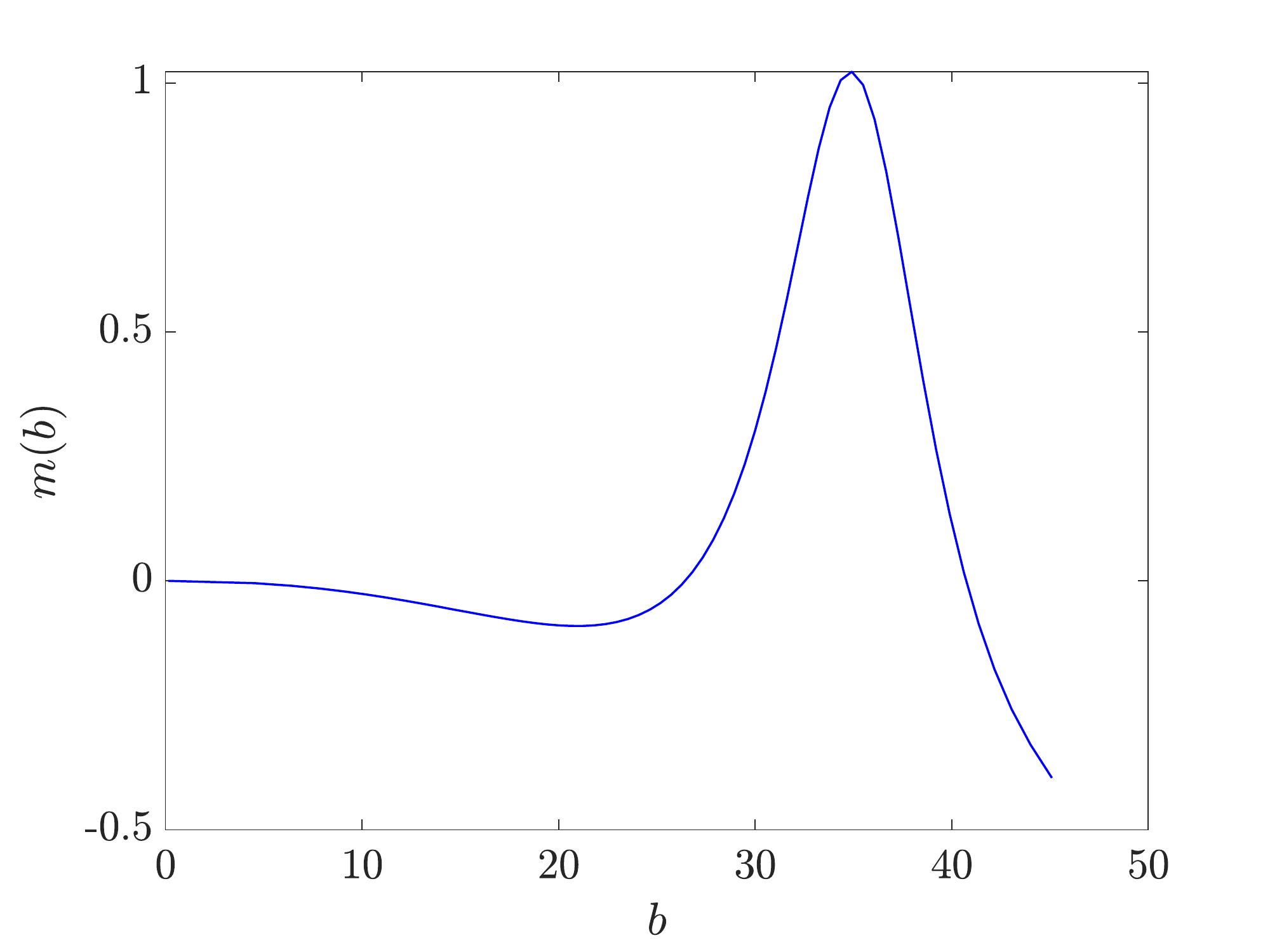}
 		\caption{Drift $m(b)$ of the Duffing oscillator for the parameters from Table~\ref{tab:CompData}.}\label{fig:exitTimeDrift}
 	\end{center}
 \end{figure}  
   
  \begin{figure}[htbp]
  	\begin{center}
  		\includegraphics[width=10cm]{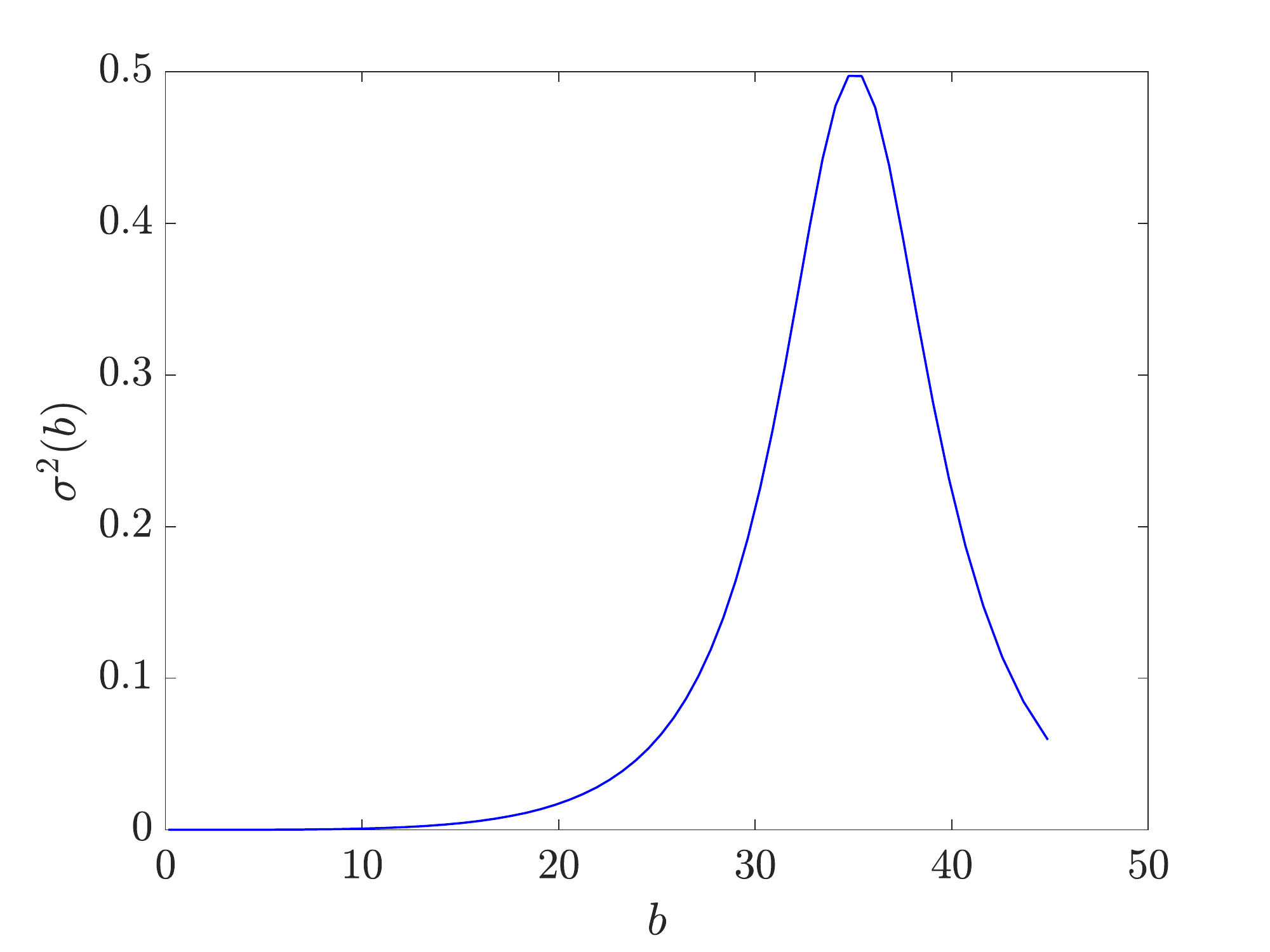}
  		\caption{Diffusion $\sigma(b)$ of the Duffing oscillator for the parameters from Table~\ref{tab:CompData}.}\label{fig:exitTimeDiffusion}
  	\end{center}
  \end{figure}

\begin{figure}[htbp]
	\begin{center}
		\includegraphics[width=11cm]{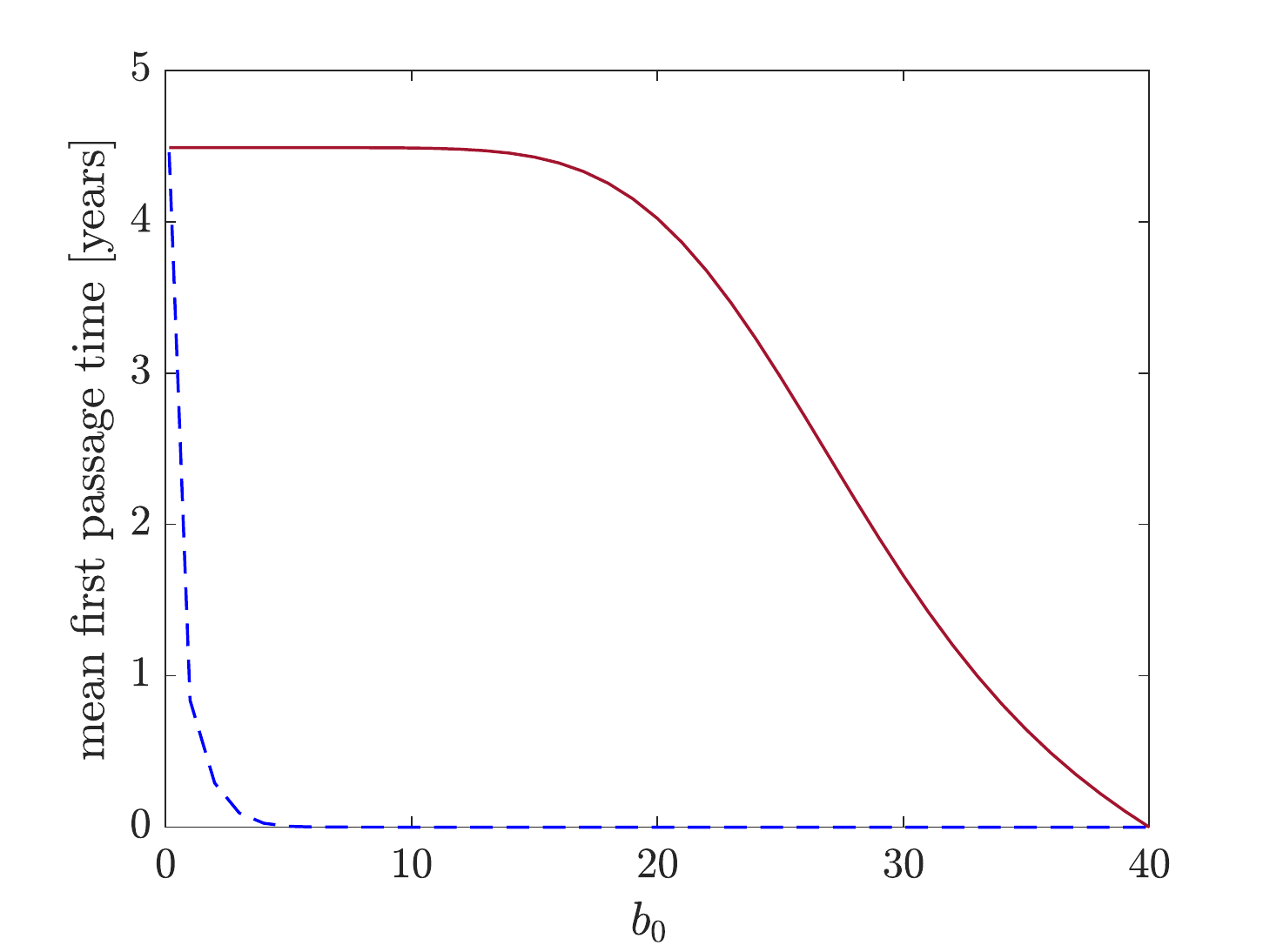}
		\caption{Comparison of mean first passage times for the Duffing oscillator for different starting oscillation amplitudes $b_0$. Solutions from Theorem~1~(\textcolor{red}{-}) and from Corollary~1(\textcolor{blue}{--  --}).}\label{fig:exitTimeDuffingComparison}
	\end{center}
\end{figure}

\section{Conclusions}
The moments of the first passage time are obtained for a one-dimensional nonlinear diffusion
processes with an entrance boundary by examining the boundary behavior. Thereby previous theorems were extended, such that they are valid for arbitrary initial points of the considered diffusion. The mean and variance of the first passage time to reach the boundary of a domain are validated with known analytical formulas, which perfectly match.
Results of the first passage times for a softening Duffing oscillator are obtained as well, which are important for the determination of dangerous ship roll dynamics in ocean waves. 
This shows, that the proposed theory is applicable for important problems. The necessary computation time is very low, since the proposed analytical expressions for the moments of the first passage times only involve integrals, which can be evaluated using standard quadrature formulas.

%\bibliography{book_lit,paper_lit,thesis_lit,ws_lit}

\begin{thebibliography}{18}
	\providecommand{\natexlab}[1]{#1}
	\expandafter\ifx\csname urlstyle\endcsname\relax
	\providecommand{\doi}[1]{doi:\discretionary{}{}{}#1}\else
	\providecommand{\doi}{doi:\discretionary{}{}{}\begingroup
		\urlstyle{rm}\Url}\fi
	
	\bibitem[{Ariaratnam \& Pi(1973)}]{ariaratnam:1973}
	Ariaratnam, S. \& Pi, H. 1973 On the first-passage time for envelope crossing
	for a linear oscillator.
	\newblock \emph{Int. J. Control}, \textbf{18}(1), 89--96.
	
	\bibitem[{Ariaratnam \& Tam(1976)}]{ariaratnam:1976}
	Ariaratnam, S. \& Tam, D. 1976 Parametric random excitation of a damped mathieu
	oscillator.
	\newblock \emph{Z. angew. Math. Mech.}, \textbf{56}, 449--452.
	
	\bibitem[{Borodin(1977)}]{borodin:1977}
	Borodin, A.~N. 1977 A limit theorem for solutions of differential equations
	with random right-hand sides.
	\newblock \emph{Theory of probability and its applications}, \textbf{22}(3),
	482--497.
	
	\bibitem[{Borodin \& Freidlin(1995)}]{FreidlinBorodin:1995}
	Borodin, A.~N. \& Freidlin, M. 1995 Fast oscillating random perturbations of
	dynamical systems with conservation laws.
	\newblock \emph{Ann. Inst. H. Poincare Probab. Statist.}, \textbf{31}(3),
	485--525.
	
	\bibitem[{Byrd \& Friedman(1954)}]{byrd:1954}
	Byrd, P.~F. \& Friedman, M.~D. 1954 \emph{Handbook of elliptic integrals for
		engineers and scientists}.
	\newblock Berlin: B. G. Teubner.
	
	\bibitem[{Dostal \emph{et~al.}(2018)Dostal, Korner, Kreuzer \&
		Yurchenko}]{dostal:2017b}
	Dostal, L., Korner, K., Kreuzer, E. \& Yurchenko, D. 2018 Pendulum energy
	converter excited by random loads.
	\newblock \emph{ZAMM - Journal of Applied Mathematics and Mechanics}.
	\newblock (\doi{10.1002/zamm.201700007})
	
	\bibitem[{Dostal \& Kreuzer(2016)}]{dostal:2016}
	Dostal, L. \& Kreuzer, E. 2016 Analytical and semi-analytical solutions of some
	fundamental nonlinear stochastic differential equations.
	\newblock \emph{Proc. IUTAM}, \textbf{19}, 178--186.
	
	\bibitem[{Dostal \emph{et~al.}(2012)Dostal, Kreuzer \&
		Sri~Namachchivaya}]{dostal:2012}
	Dostal, L., Kreuzer, E. \& Sri~Namachchivaya, N. 2012 Non-standard stochastic
	averaging of large amplitude ship rolling in random seas.
	\newblock \emph{Proceedings of the Royal Society A: Mathematical, Physical and
		Engineering Sciences}, \textbf{468}(2148), 4146--4173.
	
	\bibitem[{Freidlin \& Wentzell(2012)}]{FreidlinWentzell:2012}
	Freidlin, M. \& Wentzell, A. 2012 \emph{Random perturbations of dynamical
		systems}.
	\newblock New York: Springer-Verlag.
	
	\bibitem[{Karlin \& Taylor(1981)}]{karlin:1981}
	Karlin, S. \& Taylor, M.~H. 1981 \emph{A second course in stochastic
		processes}.
	\newblock New York: Academic Press.
	
	\bibitem[{Khasminskii(1966)}]{Khasminskii:1966}
	Khasminskii, R.~Z. 1966 A limit theorem for the solution of differential
	equations with random right-hand sides.
	\newblock \emph{Theory Probab Appl}, \textbf{11}, 390--405.
	
	\bibitem[{Khasminskii(1968)}]{Khasminskii:1968}
	Khasminskii, R.~Z. 1968 On the principles of averaging for {I}t\^{o} stochastic
	differential equations.
	\newblock \emph{Kybernetica}, \textbf{4}, 260--279.
	
	\bibitem[{Oksendal(1992)}]{Oksendal:1992}
	Oksendal, B. 1992 \emph{Stochastic differential equations (3rd ed.): An
		introduction with applications}.
	\newblock New York, NY, USA: Springer-Verlag.
	
	\bibitem[{Roberts(1978)}]{roberts1978first}
	Roberts, J. 1978 First-passage time for oscillators with nonlinear damping.
	\newblock \emph{Journal of Applied Mechanics}, \textbf{45}(1), 175--180.
	
	\bibitem[{Roberts \& Vasta(2000)}]{roberts:2000}
	Roberts, J.~B. \& Vasta, M. 2000 Markov modelling and stochastic identification
	for nonlinear ship rolling in random waves.
	\newblock \emph{Phil Trans R Soc Lond A}, \textbf{358}, 1917--1941.
	
	\bibitem[{Sri~Namachchivaya(1991)}]{namachchivaya:1991}
	Sri~Namachchivaya, N. 1991 Co-dimension two bifurcation in the presence of
	noise.
	\newblock \emph{J. appl. Mech. (ASME)}, \textbf{58}, 259--265.
	
	\bibitem[{Vanvinckenroye \& Denoel(2017)}]{vanvinckenroye2017average}
	Vanvinckenroye, H. \& Denoel, V. 2017 Average first-passage time of a
	quasi-hamiltonian mathieu oscillator with parametric and forcing excitations.
	\newblock \emph{Journal of Sound and Vibration}, \textbf{406}, 328--345.
	
	\bibitem[{Yurchenko \emph{et~al.}(2013)Yurchenko, Naess \&
		Alevras}]{Yurchenko:2013}
	Yurchenko, D., Naess, A. \& Alevras, P. 2013 Pendulum's rotational motion
	governed by a stochastic mathieu equation.
	\newblock \emph{Probabilistic Engineering Mechanics}, \textbf{31}, 12--18.
	
\end{thebibliography}
\end{document}